\documentclass{daj}

\dajAUTHORdetails{%
  title = {Tur\'an and Ramsey Problems for Alternating Multilinear Maps}, %% 
  author = {Youming Qiao},
  plaintextauthor = {Youming Qiao},
  plaintexttitle = {Turan and Ramsey Problems for Alternating Multilinear Maps}, 
  keywords = {Ramsey and Tur\'an problems, extremal combinatorics, alternating 
  multilinear maps, exterior algebras, $p$-groups, Grassmannians},
}   %%% END \dajAUTHORdetails

\dajEDITORdetails{%
   year={2023},
   number={12},
   received={14 January 2021},   % received date: example: 7 January 2017
   revised={30 March 2022},    % Optional revised date (you may comment it out)
   published={17 August 2023},  % published date
   doi={10.19086/da.84736},       % XXX = number of paper, e.g. da006 for paper#6
}   %%% END \dajEDITORdetails

\usepackage{latexsym}
\usepackage{amsmath}
\usepackage{amssymb}
\usepackage{amsthm}
\usepackage{hyperref}
\usepackage{cite}
\usepackage{graphicx}
\usepackage{color}
\usepackage{enumitem}
\usepackage{braket}
\usepackage{bm}
\usepackage{xspace}
\setlist[description]{font=\normalfont\itshape\textbullet\space}
\usepackage[all]{xy}
\theoremstyle{plain}
\newtheorem{theorem}{Theorem}[section]
\newtheorem{corollary}[theorem]{Corollary}
\newtheorem{lemma}[theorem]{Lemma}
\newtheorem{observation}[theorem]{Observation}
\newtheorem{proposition}[theorem]{Proposition}

\theoremstyle{definition}

\newtheorem{remark}[theorem]{Remark}
\newtheorem{conjecture}[theorem]{Conjecture}
\newtheorem{definition}[theorem]{Definition}

\newtheorem{question}[theorem]{Question}
\newcommand{\GL}{\mathrm{GL}}
\newcommand{\F}{\mathbb{F}}
\newcommand{\Z}{\mathbb{Z}}

\newcommand{\C}{\mathbb{C}}

\newcommand{\N}{\mathbb{N}}

\newcommand{\rk}{\mathrm{rk}}

\newcommand{\rad}{\mathrm{rad}}

\newcommand{\M}{\mathrm{M}}

\newcommand{\tuple}[1]{\mathbf{#1}}
\newcommand{\tens}[1]{\mathtt{#1}}
\newcommand{\spa}[1]{\mathcal{#1}}
\renewcommand{\span}{\mathrm{span}}

\newcommand{\cA}{\spa{A}}
\newcommand{\cB}{\spa{B}}
\newcommand{\cC}{\spa{C}}
\newcommand{\cD}{\spa{D}}

\newcommand{\tA}{\tens{A}}

\newcommand{\tD}{\tens{D}}

\newcommand{\vzero}{\tuple{0}}
\newcommand{\vstar}{{\bf *}}

\newcommand{\fB}{\mathfrak{B}}
\newcommand{\fC}{\mathfrak{C}}

\newcommand{\bA}{\mathbf{A}}

\DeclareMathOperator{\id}{id}

\newcommand{\too}%
{\xrightarrow{\text{\raisebox{-3pt}{$\sim$}}\,}}

\newcommand{\TN}{\mathrm{T}}
\newcommand{\FP}{\mathrm{FP}}
\newcommand{\CRN}{\mathrm{R}^{\mathrm{c}}}
\newcommand{\ARN}{\mathrm{R}^{\mathrm{a}}}
\newcommand{\RN}{\mathrm{R}}
\newcommand{\Gr}{\mathrm{Gr}}
\newcommand{\tr}[1]{#1^{\mathrm{t}}}
\newcommand{\size}{\mathrm{size}}

\begin{document}

\begin{frontmatter}[classification=text]
%% EDITOR: this will force the keywords to appear right after the Abstract.
%%   If the abstract is too long and would force the keywords off the
%%   front page, please comment out % [classification=text] above
%%   This way the keywords will be floated on the bottom of the first page
%%   even though the Abstract spills over to the next page.

%%% AUTHOR: Title goes here.  This line is optional.  You must use it
%%   if title has footnote attached or requires nontrivial typesetting,
%%   e.g., inclusion of linebreaks to force nice layout.
%\title{Tur\'an and Ramsey Problems \\
%for Alternating Multilinear Maps}
%\footnote{This is a footnote to the title}} %% please capitalize all significant 
%%%words

%%% AUTHOR:
%%% List all authors. If you wish, place grant acknowledgements in \thanks.
%%% In brackets include a short tag for each author.
\author[yq]{Youming Qiao\thanks{Research partially supported by Australian 
Research Council DP200100950.}}
%\author[johan]{Johan H{\aa}stad\thanks{Supported by...}}
%\author[laci]{L\'aszl\'o Lov\'asz\thanks{Supported by...}}
%\author[andy]{Andrew Chi-Chih Yao\thanks{Supported by...}}

%%% AUTHOR: Abstract goes here
\begin{abstract}
Guided by the connections between hypergraphs and exterior algebras, we study 
Tur\'an  and Ramsey type problems for 
alternating multilinear maps. This study lies at the intersection of 
combinatorics, group theory, and algebraic geometry, and has origins in the works 
of Lov\'asz (\emph{Proc. Sixth British
Combinatorial Conf.}, 1977), Buhler, Gupta, and Harris (\emph{J. Algebra}, 1987), 
and Feldman and Propp 
(\emph{Adv. Math.}, 1992).

Our main result is a Ramsey theorem for alternating bilinear maps. Given $s, t\in 
\N$, $s, t\geq 2$, and an alternating bilinear map $\phi:V\times V\to U$ with
$\dim(V)\geq s\cdot t^4$, we show that there exists either a dimension-$s$ 
subspace 
$W\leq V$ 
such 
that 
$\dim(\span(\phi(W, W)))=0$, or a dimension-$t$ subspace $W\leq V$ such that 
$\dim(\span(\phi(W, 
W)))=\binom{t}{2}$. This result has natural group-theoretic (for finite 
$p$-groups) 
and geometric 
(for Grassmannians) implications, and leads to new Ramsey-type questions for 
varieties of groups and Grassmannians. 
\end{abstract}
\end{frontmatter}

%%% AUTHOR: body of paper starts here

\section{Introduction}\label{sec:intro}

%\subsection{The main result of this paper.}
The main result of this paper is a Ramsey theorem for alternating 
bilinear maps, or equivalently, for linear spaces of alternating bilinear forms. 

To state the result we need some definitions. Let $\F$ be a field. 
Let $U$ be an $n$-dimensional vector space over $\F$. We use $W\leq U$ to 
denote that $W$ is a subspace of $U$. Recall that a bilinear form $f:U\times U\to 
\F$ is \emph{alternating}, if for any $u\in U$, $f(u, u)=0$. For $W\leq U$, the 
\emph{restriction} of $f$ to $W$ is denoted as $f|_W:W\times W\to \F$. 
Let $\Lambda(U)$ be the linear space of alternating bilinear forms on $U$. For 
$\cA\leq \Lambda(U)$ and $W\leq U$, $\cA|_W:=\{f|_W\mid f\in \cA\}\leq\Lambda(W)$. 

\begin{theorem}\label{thm:main}
Let $\F$ be a field, $s, t\in \N$, $s, t\geq 2$, and $n\geq s\cdot t^4$. Let $U$ 
be an $n$-dimensional vector space over $\F$. Then for any 
$\cA\leq\Lambda(U)$, either there exists $V\leq U$, $\dim(V)=s$, such that 
$\cA|_V$ is the zero space, or there exists $W\leq U$, $\dim(W)=t$, such that 
$\cA|_W=\Lambda(W)$. 
\end{theorem}

We will prove Theorem~\ref{thm:main} in Section~\ref{sec:thm_proof}, after 
presenting some preliminaries in Section~\ref{sec:prel}. Theorem~\ref{thm:main} 
has natural interpretations in group theory and geometry, 
and it can be naturally understood as a linear Ramsey theorem in the context of  
Tur\'an and Ramsey problems for alternating multilinear maps. These will be 
explained in Section~\ref{sec:app} and~\ref{sec:discussion}. 
A closely related 
result is 
Weaver's ``quantum'' 
Ramsey theorem \cite{Wea17}. That result concerns the so-called operator systems 
from quantum information, 
which 
are actually linear spaces of matrices over $\C$ satisfying certain conditions. 
%Such matrix 
%spaces arise from quantum information theory \cite{DSW12}. 
Cliques and anticliques 
can be defined for such matrix spaces, which are in close analogy with the 
totally-isotropic spaces and the complete spaces studied here. In
\cite{Wea17}, it was shown that 
when $n\geq 8s^{11}$, any operator system has either an $s$-clique or an 
$s$-anticlique.  

The initial strategy for our proof (in particular Steps 1 and 2, see 
Section~\ref{sec:thm_proof}) follows closely some ideas in 
Weaver's proof in \cite{Wea17} and Sims' work on enumerating $p$-groups 
\cite[Sec. 2]{Sims65}. However, new ideas are indeed required, as in the 
alternating matrix space setting, there are no diagonal matrices which are crucial 
for Weaver's proof in the operator system setting. 
Furthermore, our result works 
over any field.

%\subsection{Organisation of this paper.} The rest of this paper is organised as 
%follows. In Section~\ref{sec:thm_proof}, we 
%present 
%the proof of Theorem~\ref{thm:main}. In Section~\ref{sec:app}, we present 
%applications (or interpretations) of Theorem~\ref{thm:main} in group theory and 
%geometry. Finally, in Section~\ref{sec:discussion} we discuss this result in the 
%context of the correspondence between combinatorics and multilinear algebra.

\section{Preliminaries}\label{sec:prel} 

\subsection{Some notation} For $n\in\N$, 
$[n]:=\{1, \dots, n\}$. The set of size-$\ell$ subsets of $[n]$ is denoted as 
$\binom{[n]}{\ell}$. 
%The natural total order of $[n]$ induces the lexicographic 
%order 
%on $\binom{[n]}{\ell}$. 

\subsection{Vector spaces} For a field $\F$, $\F^n$ is the linear space 
consisting of length-$n$ 
column vectors over $\F$. We use $e_i$ to denote the $i$th standard basis vector 
in $\F^n$. 
The dual space of $\F^n$ is denoted as $(\F^n)^*$, and the dual vector 
of $v\in \F^n$ is denoted as $v^*$. 
For $S\subseteq\F^n$, $\span(S)$ 
denotes the subspace spanned by vectors in $S$. For $v\in \F^n$ and $i\in[n]$, 
$v(i)$ denotes the $i$th component of $v$. 
For $S\subseteq\F^n$, 
$S^\perp:=\{v\in \F^n \mid \forall u\in S, \tr{v}u=0\}$.

\subsection{Matrices} 
We use $\M(\ell\times n, \F)$ to denote the linear space 
of $\ell 
\times n$ 
matrices over $\F$, and set $\M(n, \F):=\M(n\times n, \F)$. For $A\in\M(\ell\times 
n, \F)$, $A(i,j)$ is the $(i,j)$th entry of $A$. We shall use $\vzero$ to denote 
all-zero 
vectors or matrices of appropriate sizes.
A matrix $A\in \M(n, 
\F)$ 
is \emph{alternating}, if for any $v\in \F^n$, $\tr{v}Av=0$. 
%The linear space of 
%$n\times n$ alternating matrices over $\F$ is denoted by $\Lambda(n, \F)$. 
For 
$(i, j)\in[n]\times [n]$, $E_{i,j}\in \M(n, 
\F)$ is 
the $n\times n$ matrix with the $(i, j)$th entry being $1$, and the rest entries 
being $0$. For 
$\{i,j\}\in \binom{[n]}{2}$, $i<j$, $A_{i,j}\in 
\M(n, \F)$ is the $n\times n$ alternating matrix with the $(i,j)$th entry being 
$1$, the 
$(j,i)$th entry being $-1$, and the rest entries being $0$. 
The 
general linear group of degree $n$ over $\F$ is denoted by $\GL(n, 
\F)$.

\subsection{Alternating matrix spaces}\label{subsec:alt_mat_sp}
Let
$\Lambda(n, \F)$ be 
the linear space of $n\times n$ alternating matrices over $\F$. Subspaces of 
$\Lambda(n, \F)$ are called alternating matrix spaces. 

%We introduce some 
%basic notions for 
%alternating matrix spaces. 

Two alternating matrix spaces $\cA, \cB\leq \Lambda(n, \F)$ are \emph{isometric}, 
if there exists $T\in\GL(n, \F)$, such that 
$\cA=\tr{T}\cB T:=\{\tr{T}BT \mid B\in\cB\}$.

Let $\cA\leq\Lambda(n, \F)$. Suppose $U\leq \F^n$ is of dimension $d$, 
%Suppose we have a $d$-dimensional $U\leq \F^n$ and $\cA\leq \Lambda(n, \F)$. 
%We say that $U$ is a 
%\emph{totally-isotropic space} for $\cA$, if for any $u, u'\in U$ and $A\in \cA$, 
%we have 
%$\tr{u}Au'=0$. 
%Suppose $\dim(U)=d$. 
%We say that $U$ is an \emph{anisotropic} space for $\cA$, if for any 
%linearly independent $u, u'\in U$, there exists some $A\in \cA$, such that 
%$u^tAu'\neq 0$. 
and let $T\in \M(n\times d, \F)$ be a matrix whose column 
vectors span $U$. The restriction of $\cA$ to $U$ via $T$ is $\cA|_{U, 
T}:=\{\tr{T}AT \mid A\in \cA\}\leq \Lambda(d, \F)$. Given another $T'\in\M(n\times 
d,\F)$ whose columns 
also 
span $U$, $\cA|_{U, T}$ and $\cA|_{U, T'}$ are isometric. Therefore, we may write 
$\cA|_U$ as the restriction of $\cA$ to $U$ via some $T\in\M(n\times d, \F)$ whose 
columns span $U$.

We say that $U$ is a \emph{totally-isotropic space}\footnote{Totally-isotropic 
spaces are also called totally singular \cite{Atk73}, isotropic \cite{BGH87}, 
and $\ell$-singular 
\cite{DS10} (for alternating $\ell$-linear maps) in the literature. We 
adopt the terminologies of totally-isotropic and 
anisotropic (see 
Section~\ref{subsec:opp_struct}), which are from the study of bilinear forms 
\cite{MH73}, and have been used in e.g. \cite{BMW17}. } for $\cA$, if 
$\dim(\cA|_U)=0$. That is, for any $u, u'\in U$ and $A\in \cA$, we have 
$\tr{u}Au'=0$. 
We say that $U$ is a \emph{complete space} for $\cA$, if 
$\dim(\cA|_U)=\binom{d}{2}$. 

Given $v\in \F^n$, the \emph{degree} of $v$ in $\cA$, $\deg_\cA(v)$, is the 
dimension of $\{ Av \mid A\in \cA\}\leq \F^n$. When $\cA$ is clear from the 
context, 
we may simply write $\deg(v)$ instead of $\deg_\cA(v)$. The \emph{minimum degree} 
of $\cA$, denoted as $\delta(\cA)$, is the minimum over the degrees over non-zero 
$v$ in $\cA$. 

Given $S\subseteq \F^n$, the \emph{radical space} of $S$ in $\cA$ is 
$\rad_\cA(S):=\{u\in \F^n \mid \forall A\in \cA, v\in S, \tr{u}Av=0\}$. We may 
simply 
write $\rad_\cA(\{v\})$ as $\rad_\cA(v)$. When $\cA$ is 
clear from the context, we may simply write $\rad(S)$ instead of $\rad_\cA(S)$. We 
also define the \emph{radical space} of $\cA\leq \Lambda(n, \F)$ as
$\rad(\cA):=\rad_\cA(\F^n)=\{v\in \F^n \mid \forall A\in\cA, Av=0\}$.

\subsection{Alternating multilinear maps} An $\ell$-linear map 
$\phi:(\F^n)^{\times \ell}\to \F^m$ is \emph{alternating}, if for any $(v_1, 
\dots, 
v_\ell)\in (\F^n)^{\times \ell}$ where $v_i=v_j$ for some $i\neq j$, $\phi(v_1, 
\dots, v_\ell)=0$.
Here, $(\F^n)^{\times \ell}$ denotes 
the $\ell$-fold Cartesian product of $\F^n$. Two alternating $\ell$-linear maps 
$\phi, \psi: (\F^n)^{\times \ell}\to \F^m$ are \emph{isomorphic}, if there exists 
$(S, 
T)\in\GL(n, \F)\times\GL(m, \F)$, such that for any $v_1, \dots, v_\ell\in \F^n$, 
$\phi(S(v_1), \dots, S(v_\ell))=T(\phi(v_1, \dots, v_\ell))$.

%
%We also need the language of $3$-way arrays and alternating matrix spaces, which 
%will be useful in proving Theorem~\ref{thm:main}.

%We use $\M(n\times d, \F)$ to denote the linear space of $n\times d$ matrices 
%over 
%$\F$, and $\M(n, \F)=\M(n\times n, \F)$. A matrix $A\in \M(n, \F)$ is 
%\emph{alternating}, 
%if for any $v, v'\in \F^n$, $v^tAv'=0$. We use $\Lambda(n, \F)$ to denote the 
%linear space of $n\times n$ alternating matrices over $\F$. Subspaces of 
%$\Lambda(n, \F)$ are referred to as alternating matrix spaces. 

\subsection{3-way arrays} Matrices are 2-way arrays, i.e. an 
array with two indices. We shall also need the notion of \emph{3-way arrays}, 
namely arrays with three indices. We use $\M(\ell\times m\times 
n, \F)$ to denote the linear space of 3-way arrays with the index set being 
$[\ell]\times[m]\times [n]$. Let $\tA\in\M(\ell\times m\times n, \F)$ be a 3-way 
array. Following \cite{KB09,GQ19}, we define the following. The \emph{frontal 
slices} of $\tA$ are $A_1, 
\dots, A_n\in \M(\ell\times m)$, where $A_k(i, j)=\tA(i,j,k)$. The \emph{tube 
fibres} of $\tA$ are $v_{i,j}\in \F^n$, $i\in[\ell]$, $j\in[m]$, where 
$v_{i,j}(k)=\tA(i,j,k)$.

3-way arrays are also referred to as 3-tensors in 
some literature. We adopt 3-way arrays, because 3-tensors are usually considered 
to be $3$-way arrays together with 
the natural action of $\GL(\ell, \F)\times\GL(m, \F)\times\GL(n, \F)$. In this 
paper, as the reader will see soon, 3-way arrays are used to record the structure 
constant of alternating 
bilinear maps. Therefore, the group action of relevance in this context is by 
$\GL(n, \F)\times 
\GL(m, \F)$ 
where $\GL(n, \F)$ acts covariantly on the first two indices. 
%we shall mostly study 3-way arrays of size 
%$n\times n\times m$ with 

\subsection{Relations between 3-way arrays, bilinear maps, and matrix spaces} 
\label{subsec:rel}
It is not hard to see that alternating bilinear maps, alternating matrix spaces, 
and 3-way arrays are closely related. We spell out some details here.

An alternating bilinear map 
$\phi:\F^n\times 
\F^n\to 
\F^m$ can be represented as a tuple of alternating matrices 
$\bA=(A_1, \dots, A_m)\in\Lambda(n, \F)^m$, such that for any $u, v\in \F^n$, 
$\phi(u, v)=\tr{(\tr{u}A_1v, \dots, \tr{u}A_mv)}$.

From an alternating matrix tuple $\bA\in\Lambda(n, \F)^m$, we can construct a 
$3$-way array $\tA\in\M(n\times n\times m, \F)$ whose frontal slices are $A_i$'s. 
We can also construct an alternating matrix space $\cA=\span\{A_1, \dots, 
A_m\}\leq\Lambda(n, \F)$. 

Let $\cA\leq\Lambda(n, \F)$ be an $m$-dimensional alternating matrix space. Let
$\bA=(A_1, \dots, A_m)\in\Lambda(n, \F)$ be an ordered linear basis of $\cA$. Then 
$\bA$ gives rise to an alternating bilinear map $\phi$ and a 3-way array $\tA$ as 
above. Note that different ordered bases of $\cA$ yield different but isomorphic
alternating bilinear maps.

\section{Proof of Theorem~\ref{thm:main}}\label{sec:thm_proof}

%\subsection{The proof} We are now ready to prove Theorem~\ref{thm:main}.
%The first two steps are based on some ideas from \cite[Theorem 4.5]{Wea17}.

\subsection{Restatement of Theorem~\ref{thm:main}} By fixing a basis for $U$ and 
identifying alternating bilinear forms with alternating matrices, 
Theorem~\ref{thm:main} can be restated as follows. 

\paragraph{Theorem~\ref{thm:main}, restated.} Let $\F$ be a field, $s, t\in \N$, 
$s, t\geq 2$, and $n\geq s\cdot t^4$. For any alternating matrix space
$\cA\leq\Lambda(n, \F)$, either there exists an $s$-dimensional totally-isotropic 
space of $\cA$, or there exists a $t$-dimensional complete space of $\cA$.

%The perspective of $3$-way arrays is also very useful, as it will be important to 
%manipulate the 
%tube fibres 
%of $\tA$. 

\subsection{Proof outline} 
To start with, by restricting to any subspace of dimension $s\cdot t^4$, we can 
assume that $\cA\leq\Lambda(n, \F)$ where $n=s\cdot t^4$. 

The proof consists of four steps. Before going into the 
details, let us first outline the objective for each step. 

%Recall that 
%we have $\cA\leq\Lambda(n, \F)$ and $n=s\cdot t^4$.

\begin{description}
\item[Step 1] This step either constructs a $s$-dimensional 
totally-isotropic space of $\cA$, or computes a $n'$-dimensional $P\leq \F^n$, 
such that the minimum degree $\delta(\cA|_P)\geq t^4$. 

Let $\cB=\cA|_P\leq\Lambda(n', \F)$. The goal of the next three steps is to 
construct a dimension-$(t+1)$ complete space for $\cB$.

\item[Step 2] Let $t'=t^2$. This step
constructs a 
$Q_2\in \M(n'\times (t'+1), \F)$, of rank-$(t'+1)$, such that $\cC=\tr{Q_2}\cB 
Q_2\leq 
\Lambda(t'+1, \F)$ (the restriction of $\cB$ to the subspace of $\F^{n'}$ spanned 
by the columns of $Q_2$) contains matrices $C_1, \dots, C_{t'}\in \cC$, where
$C_i=\begin{bmatrix}
\tilde C_i & \vzero \\
\vzero & \vzero 
\end{bmatrix}$, $\tilde C_i$ has size $(i+1)\times (i+1)$, and $C_i(i, 
i+1)=1$.

\item[Step 3] Let $r=t+\binom{t}{2}$, and recall that $t'=t^2$. This step 
constructs
$Q_3\in\GL(t'+1, \F)$, such that $\cD=\tr{Q_3}\cC Q_3\leq 
\Lambda(t'+1, 
\F)$ contains 
$D_1, 
\dots, D_r\in \Lambda(t'+1, \F)$ satisfying (1) for $i\in [t]$, 
$D_i=\begin{bmatrix}
\tilde D_i & \vzero \\
\vzero & \vzero 
\end{bmatrix}$ where $\tilde D_i\in \Lambda(i+1, \F)$ and $D_i(i, i+1)=1$,  
and 
(2) for 
$i\in[\binom{t}{2}]$, $D_{t+i}=\begin{bmatrix}
\tilde D_{t+i} & \vzero & \vzero & \vzero \\
\vzero & 0 & 1 & \vzero \\
\vzero & -1 & 0 & \vzero \\
\vzero & \vzero & \vzero & \vzero 
\end{bmatrix}$, where $\tilde D_{t+i}$ is of size $t+1+2(i-1)$.

\item[Step 4] 
Based on $D_1, \dots, D_r$ from Step 3, this step constructs $Q_4\in\GL(t'+1, 
\F)$, such that $W=\span\{ e_1, \dots, e_{t+1}\}$ is a complete space for 
$\tr{Q_4}\cD Q_4$. 
It is clear that the complete space $W$ of $\tr{Q_4}\cD Q_4$ translates to a 
complete space for $\cA|_P$ through $Q_2$, $Q_3$, and $Q_4$, giving us the desired 
complete 
space for $\cA$.
%
%Let $\tD\in\M((t'+1)\times 
%(t'+1)\times r, \F)$ be the 3-way array whose $i$th frontal slice is 
%$D_i$ from Step 3. Let $f_{i, 
%j}\in \F^r$ be the tube fibre at the $(i, j)$th position of $\tD$; that is, 
%$f_{i,j}(k)=\tD(i,j,k)=D_k(i,j)$. From Step 3, $f_{i, 
%i+1}$, $i\in[t]$, and $f_{t+2j, t+2j+1}$, $j\in[\binom{t}{2}]$, are linearly 
%independent. 
%The goal in this step is to perform row and column operations on 
%$\tD$ to make the 
%tube fibres at the positions $(i, j)$, $1\leq i<j\leq t+1$, to be linearly 
%independent. 
%By Observation~\ref{obs:tube}, this would result in a complete space 
%of 
%dimension $t+1$. 
\end{description}

Each step relies on a lemma which could be of independent interest. In the 
following, we explain these steps in detail.

\subsection{Step 1} 
%In Section~\ref{subsec:alt_mat_sp} we defined the minimum degree of 
%$\cA\leq\Lambda(n, \F)$, denoted as $\delta(\cA)$. 
The first step 
relies on the following lemma. 
Recall that the minimum degree $\delta(\cdot)$ is defined in 
Section~\ref{subsec:alt_mat_sp}.
\begin{lemma}\label{lem:iso_min-deg}
Let $s, d\in \N$, $s, d\geq 2$, and $n=s\cdot d$. For any $\cA\leq\Lambda(n, 
\F)$, either there exists a totally-isotropic space of 
dimension $s$ of $\cA$, or there exists $P\leq\F^n$, such that $\dim(P)\geq 2d$, 
and $\delta(\cA|_P)\geq 
d$. 
\end{lemma}

\begin{proof} %[Proof of Lemma~\ref{lem:iso_min-deg}.] 
Consider the following 
procedure in at most $s$ rounds. The basic 
idea is that in each round, if there exists a non-zero vector $v$ of degree $<d$, 
we restrict the current alternating matrix space to $\rad(v)$.

More specifically, in the first round, if there are no non-zero $v_1\in\F^n$ 
such that $\deg_\cA(v_1)<d$, then
$\F^n$ satisfies what we need for $P$. Otherwise, there exists a non-zero 
$v_1\in\F^n$ 
such that $\deg_\cA(v_1)<d$. Let $S_1=\span\{ v_1\}$, 
$T_1=\rad_\cA(S_1)$, and $\cA_1=\cA|_{T_1}$. By the alternating 
property, $v_1\in T_1$, so $S_1\leq T_1$. Make $R_1$ a complement subspace 
of $S_1$ in $T_1$. 
%
%decide if there exists a 
%non-zero $v_1\in \F^n$ such that $\deg_\cA(v_1)<d$. There are two possibilities:
%\begin{itemize}
%\item If no such $v_1$ exist, then
%$\F^n$ satisfies what we need for $P$. 
%\item Otherwise, let $S_1=\langle v_1\rangle$, 
%$T_1=\rad_\cA(S_1)$, and $\cA_1=\cA|_{T_1}$. Note that by the alternating 
%property, $v_1\in T_1$, so $S_1\leq T_1$. Set $R_1$ to be a complement subspace 
%of $S_1$ in $T_1$. 
%\end{itemize}
By $\deg_\cA(v_1)<d$, we have $\dim(T_1)\geq 
(s-1)d+1$ and $\dim(R_1)\geq (s-1)d$. Also note that $S_1\leq \rad(\cA_1)$. We 
then 
continue to 
the next round.

%In the second 
%round, decide if there exists $v_2\in R_1$ such that $\deg_{\cA_1}(v_2)<d$. 
%Again, there are two possibilities:
%\begin{itemize}
%\item If no such 
%$v_2$ exist, then $R_1$ satisfies what we need 
%for $P$. 
%\item Otherwise, let 
%$S_2=\langle v_1, 
%v_2\rangle$, 
%$T_2=\rad_{\cA_1}(S_2)=\rad_{\cA_1}( v_2)$, and 
%$\cA_2=\cA_1|_{T_2}=\cA|_{T_2}$. Set $R_2$ to be a complement subspace of $S_2$ 
%in 
%$T_2$. 
%
%By $\deg_{\cA_1}(v_2)<d$, we have $\dim(T_2)\geq (s-2)d+2$ and $\dim(R_2)\geq 
%(s-2)d$. Also note that $S_2\leq \rad(\cA_2)$, so in particular, $S_2$ is a 
%totally-isotropic space of $\cA$. We 
%then continue 
%to the next round.
%\end{itemize}

Then before the $i$th round, $i=2, \dots, s-1$, we have obtained from the 
$(i-1)$th round the following data.
\begin{itemize}
\item $S_{i-1}=\span\{ v_1, \dots, v_{i-1}\}\leq \F^n$, $\dim(S_{i-1})=i-1$, 
and $S_{i-1}$ is a totally-isotropic space of $\cA$.
\item $T_{i-1}\leq \F^n$, $\dim(T_{i-1})\geq (s-(i-1))d+(i-1)$, and $S_{i-1}\leq 
T_{i-1}$.
\item $R_{i-1}$, a complement subspace of $S_{i-1}$ in $T_{i-1}$. Note that 
$\dim(R_{i-1})\geq (s-(i-1))d\geq (s-(s-1-1))d=2d$.
\item $\cA_{i-1}=\cA|_{T_{i-1}}$, and $S_{i-1}\leq \rad(\cA_{i-1})$.
\end{itemize}
%Then, in the $i$th round, $i=1, \dots, d-1$, 

We then try to find a non-zero $v_i\in 
R_{i-1}$ such that 
$\deg_{\cA_{i-1}}(v_i)<d$. 
%There are two possibilities.
%\begin{itemize}
%\item 
If no such $v_i$ exist, then $R_{i-1}$ satisfies what 
we need for $P$. 
%\item 
Otherwise, let $S_i=\span\{ v_1, \dots, v_i\}$, 
$T_i=\rad_{\cA_{i-1}}(S_i)=\rad_{\cA_{i-1}}( v_i)$, 
and $\cA_i=\cA|_{T_i}$. Set $R_i$ to be a complement subspace of $S_i$ in $T_i$.
As $\deg_{\cA_{i-1}}(v_i)<d$, $\dim(T_i)\geq (s-i)d+i$, and $\dim(R_i)\geq 
(s-i)d$. Clearly $S_i\leq\rad(\cA_i)$, so in particular, $S_i$ is a 
totally-isotropic space for $\cA$.
%As $v_i\in T_{i-1}\leq \rad_{\cA_{i-1}}(S_{i-1})$, $S_i\leq\rad(\cA_i)$. 
%\end{itemize}
We then continue 
to the 
$(i+1)$th round. 

Now suppose we just enter the $s$th round. At this point, we have 
$\dim(T_{s-1})\geq 
d+(s-1)$, and $\dim(R_{s-1})\geq d$. 
%and 
%$U_{s-1}=\rad_{\cA_{s-1}}(S_{s-1})$. 
%Since $\dim(U_{s-1})>d\geq s>S_{s-1}$, 
Take 
any non-zero $v_s\in R_{s-1}$, and set 
$S_s=\span\{ v_1, \dots, v_s\}$. Since $S_{s-1}\leq\rad(\cA_{s-1})$, 
$S_s$ is 
a totally-isotropic space of 
dimension $s$.
\end{proof}

Back to our original setting, recall that $n=s\cdot t^4$. Let $d=t^4$, so 
$n=s\cdot d$. Applying Lemma~\ref{lem:iso_min-deg} gives us 
either 
a totally-isotropic 
space of 
dimension $s$, or a subspace $P$ such that $n'=\dim(P)\geq 2d$ and 
$\delta(\cA|_P)\geq d$. In the former case, we can conclude the proof of 
Theorem~\ref{thm:main}. 
In the latter case, we will construct a 
dimension-$(t+1)$ totally-isotropic space for $\cB$ in the next three steps.

%If the above procedure does not lead to the $s$th round, it means that there 
%exists some $R_i$, 
%$i\in[s-2]$, such that $\cA|_{R_i}$ is of minimum degree $\geq d$. Note that 
%$\dim(R_i)\geq (s-i)d\geq 2d$. Let $P$ be this $R_i$, $n'=\dim(P)$, and 
%$\cB=\cA|_P\leq\Lambda(n', \F)$. 

\subsection{Step 2}  The second step relies on the following lemma.

\begin{lemma}\label{lem:step2}
Let $d, t'\in \N$, $d=t'^2$. Suppose $\cB\leq\Lambda(n', \F)$ satisfies that 
$n'\geq 2d$ and $\delta(\cB)\geq d$. Then there exists $Q\in \M(n'\times (t'+1), 
\F)$, of rank-$(t'+1)$, such that $\cC=\tr{Q}\cB 
Q\leq 
\Lambda(t'+1, \F)$ contains matrices $C_1, \dots, C_{t'}\in \cC$, where
$C_i=\begin{bmatrix}
\tilde C_i & \vzero \\
\vzero & \vzero 
\end{bmatrix}$, $\tilde C_i$ has size $(i+1)\times (i+1)$, and $C_i(i, 
i+1)=1$.
\end{lemma}
\begin{proof}
%
%The objective 
%of Step 2 is to construct a rank-$(t'+1)$ matrix $Q\in \M(n'\times (t'+1), \F)$, 
%such that the following holds. Let $\cC=\tr{Q}\cB 
%Q\leq 
%\Lambda(t'+1, \F)$. Note that $\cC$ is the restriction of $\cB$ to the subspace 
%of 
%$\F^{n'}$ spanned by the columns of $Q$. Then there exist $C_1, \dots, C_{t'}\in 
%\cC$, such that 
%$C_i=\begin{bmatrix}
%C_i' & \vzero \\
%\vzero & \vzero 
%\end{bmatrix}$, where $C_i'$ is of size $(i+1)\times (i+1)$, and $C(i, i+1)=1$.

%we will construct a tuple of 
%vectors $(w_1, \dots, w_{t+1})$, $w_i\in \F^{n'}$, and a tuple of matrices $(B_1, 
%\dots, B_{t})$, $B_i\in \cB$, such that they satisfy the following. First, 
%$w_i$'s are linearly independent. Second, for $i\in[t]$, $w_i^tB_iw_{i+1}\neq 0$, 
%and $w_j^tB_iw_k=0$ for any $j$ and any $k>i+1$. 

%Let $t=s^2$. 
Consider the following procedure in at most $t'$ 
rounds. 

In the first round, take any non-zero $w_1\in \F^{n'}$, and find $B_1\in \cB$, 
such that $B_1w_1\neq\vzero$. Then there exists $w_2\in\F^{n'}$
such that $\tr{w_2}B_1w_1\neq 0$. Such $B_1$ exists, as $\dim(\rad_{\cB}(w_1))\leq 
n'-d$. Set $T_2=(\span\{ B_1w_1, B_1w_2\})^\perp$, and 
$W_2=\span\{ w_1, w_2\}$. By $\tr{w_2}B_1w_1\neq 0$, $T_2\cap W_2=0$. By the 
alternating property, $\dim(W_2)=2$. 

%We then enter the second round. We claim that there exist a non-zero $w_3\in T_2$ 
%and
%$B_2\in \cB$, such that $\tr{w_3}B_2w_2\neq 0$. To see this, note that 
%$\dim(T_2)\geq 
%n'-2$, and $\rad_{\cB}(w_2)\leq n'-d$. So as long as $d>2$, we can take any 
%$w_3\in T_2\setminus 
%\rad_{\cB}(w_2)$, for which there exists $B_2\in \cB$ such that 
%$\tr{w_3}B_2w_2\neq 
%0$. Note that $B_2$ and $B_1$ are linearly independent, as $\tr{w_3}B_1w_2=0$.
%
%We then set $T_3=\langle B_iw_j : i=1, 2, j=1, 2, 3\rangle^\perp$, and 
%$W_3=\langle w_1, w_2, w_3\rangle$. 
%We claim that $W_3\cap T_3=0$. If not, suppose 
%$w=\alpha_1w_1+\alpha_2w_2+\alpha_3w_3\in W_3\cap T_3$, and $i$ is the smallest 
%integer such that $\alpha_i$ is non-zero. If $i=1$, then 
%$\tr{w}B_1w_2=\alpha_1\tr{w_1}B_1w_2\neq 0$. If 
%$i=2$, then $\tr{w}B_1w_1=\alpha_2\tr{w_2}B_1w_1\neq 0$. If $i=3$, then 
%$\tr{w}B_2w_2=\alpha_3\tr{w_3}B_2w_2\neq 
%0$. In either of these three cases, it is a contradiction to the assumption that 
%$w\in T_3$. 
%We also note that, by examining 
%$w_3^tB_iw_2$, we have $B_2\not\in\langle B_1\rangle$.
%$B_3\not\in\langle B_1, B_2\rangle$. 

Then before the $i$th round, $i=2, \dots, t'$, we have obtained $w_1, 
\dots, 
w_i\in \F^{n'}$, $B_1, 
\dots, B_{i-1}\in \cB$, 
such that $\forall j\in[i-1]$, (1) $\tr{w_j}B_jw_{j+1}\neq 0$, and (2) $\forall 
1\leq 
k\leq i$, $\forall j+1<\ell\leq i$, $\tr{w_k}B_jw_\ell=0$ . We 
also have $T_i=(\span\{ B_jw_k \mid j\in[i-1], k\in[i]\})^\perp$, and $W_i=\span\{ 
w_1, \dots, w_i\}$, such that $T_i\cap W_i=0$. 

We claim that there exist 
$w_{i+1}\in T_i$ and $B_i\in \cB$, such that $\tr{w_{i+1}}B_iw_i\neq 0$. To see 
this, note that $\dim(T_i)\geq n'-(i-1)\cdot i$ and $\dim(\rad_\cB(w_i))\leq 
n'-d$. So as long as $d>(i-1)\cdot i$, we can take any $w_{i+1}\in T_i\setminus 
\rad_{\cB}(w_i)$, for which there exists $B_i\in \cB$ such that 
$\tr{w_{i+1}}B_iw_i\neq 0$. 

Let $T_{i+1}=(\span\{ B_jw_k \mid j\in[i], k\in[i+1]\})^\perp$, and 
$W_{i+1}=\span\{ w_1, \dots, w_{i+1}\}$. We claim that $T_{i+1}\cap 
W_{i+1}=0$. If not, suppose 
$w=\alpha_1w_1+\dots+\alpha_iw_i+\alpha_{i+1}w_{i+1}\in 
T_{i+1}$. Let $j$ be the smallest integer such that $\alpha_j\neq 0$. If $j\leq 
i$, then $\tr{w}B_jw_{j+1}$ is non-zero. If $j=i+1$, then $\tr{w}B_iw_i$ is 
non-zero. In 
either case, this is a contradiction to the assumption that $w\in T_{i+1}$. We 
also note that, by examining $\tr{w_{i+1}}B_iw_i$, we have $B_i\not\in\span\{ B_1, 
\dots, B_{i-1}\}$. 

We perform the above operations, and after the $t'$-th round we get the desired 
$w_1, 
\dots, w_{t'+1}$ and $B_1, \dots, B_{t'}\in\Lambda(n', \F)$. This requires 
$d>(t'-1)t'$, which is 
fine as we have set $d=t'^2$. 

Let $Q=\begin{bmatrix} w_1 & \dots & 
w_{t'+1}\end{bmatrix}\in\M(n'\times(t'+1),\F)$, 
and let $\cC=\tr{Q}\cB Q\leq\Lambda(t'+1, \F)$. We claim that $\cC$ satisfies the 
requirements of this lemma. Recall that for any $i\in[t']$, $B_i$ satisfies that 
$\tr{w_i}B_iw_{i+1}\neq 0$. Furthermore, for any $i+1<\ell\leq t+1$ and $1\leq 
k\leq 
t+1$, $\tr{w_k}B_iw_\ell=0$. Set $\tr{w_i}B_iw_{i+1}=\alpha_i$. Let
$C_i=\tr{Q}(\frac{1}{\alpha_i} B_i)Q\in \cC$. Then $C_i(i, 
i+1)=\frac{1}{\alpha_i}\tr{w_i}B_iw_{i+1}=1$, and for any $i+1<\ell\leq t+1$ and 
$1\leq k\leq t+1$, $C_i(k, \ell)=\frac{1}{\alpha_i}\tr{w_k}B_iw_\ell=0$. 
Intuitively, this just 
means 
that $C_i$ is of the form $\begin{bmatrix}
\tilde C_i & \vzero\\
\vzero & \vzero
\end{bmatrix}$, where $\tilde C_i$ is of size $(i+1)\times (i+1)$, and $C_i(i, 
i+1)=1$. 
\end{proof}

From Step 1, we have $\cB\leq\Lambda(n', \F)$ with 
$n'\geq 
2d$ and $\delta(\cB)\geq d$. Recall that $d=t^4$, and set $t'=t^2$. Applying 
Lemma~\ref{lem:step2} to $\cB$, $d$, $t'$ produces $Q_2\in\M(n'\times(t'+1), \F)$ 
which achieves the objective of this step. 

\subsection{Step 3} The objective of this step is fulfilled by the following 
lemma. 
\begin{lemma}\label{lem:step3}
Let $t'=t^2$, and $r=t+\binom{t}{2}$. 
Suppose $\cC\leq 
\Lambda(t'+1, \F)$ contains matrices $C_1, \dots, C_{t'}\in \cC$, where
$C_i=\begin{bmatrix}
\tilde C_i & \vzero \\
\vzero & \vzero 
\end{bmatrix}$, $\tilde C_i$ has size $(i+1)\times (i+1)$, and $C_i(i, 
i+1)=1$.
Then there exists $Q\in\GL(t'+1, \F)$, such that $\cD=\tr{Q}\cC Q\leq 
\Lambda(t'+1, 
\F)$ contains 
$D_1, 
\dots, D_r\in \Lambda(t'+1, \F)$ satisfying (1) for $i\in [t]$, 
$D_i=\begin{bmatrix}
\tilde D_i & \vzero \\
\vzero & \vzero 
\end{bmatrix}$ where $\tilde D_i\in \Lambda(i+1, \F)$ and $D_i(i, i+1)=1$, 
and 
(2) for 
$i\in[\binom{t}{2}]$, $D_{t+i}$ is of the form
%$D_{r+i}=\begin{bmatrix}
%D_i' & \vzero & \vzero & \vzero \\
%\vzero & 0 & 1 & \vzero \\
%\vzero & -1 & 0 & \vzero \\
%\vzero & \vzero & \vzero & \vzero 
%\end{bmatrix}$, where $D_i'$ is of size $s+2(i-1)$.
\begin{equation}\label{eq:form_of_D}
\begin{bmatrix}
\tilde D_{t+i} & \vzero & \vzero & \vzero & \dots & \vzero \\
\vzero & 0 & 1 & 0 & \dots & 0 \\
\vzero & -1 & 0 & 0 & \dots & 0 \\
\vzero & 0 & 0 & 0 & \dots & 0 \\
\vdots & \vdots & \vdots & \vdots & \ddots & \vdots \\
\vzero & 0 & 0 & 0 & \dots & 0 
\end{bmatrix}, 
\end{equation}
where $\tilde D_{t+i}\in\Lambda(t+1+2(i-1), \F)$. 
\end{lemma}

That is, in the matrix $D_{t+i}$, the only nonzero 
entries 
of the $(t+2i)$th and $(t+2i+1)$th rows and columns are at the $(t+2i, t+2i+1)$ 
and $(t+2i+1, t+2i)$ positions. The reason for imposing this condition will be 
clear later from Observation~\ref{obs:key}.

\begin{proof}[Proof of Lemma~\ref{lem:step3}]
Observe that the $(t+2i, t+2i+1)$ and 
$(t+2i+1, t+2i)$ entries of $C_{t+2i}$ are $1$ and $-1$, respectively. So we can 
put $C_{t+2i}$ in the desired form, by multiplying appropriate elementary 
matrices  
(i.e. $I+\alpha\cdot E_{t+2i, j}$ for $j<t+2i$ and appropriate $\alpha\in\F$) on 
the left and their transposes on the right, 
to set other entries on the $(t+2i+1)$th row and 
column to be $0$. 
%row and column operations easily. 
Some 
care is required to ensure that during 
this process, other $C_{t+2j}$'s, if they are already in this form, are not 
affected. 

Therefore, we apply appropriate matrices to 
%perform row and column operations for 
$C_{t+2i}$, 
for $i$ in a \emph{decreasing} order, namely starting from $i=\binom{t}{2}$ and 
then 
going to $i=1$. To see that this does not affect those $C_{t+2j}$'s which were 
already in this form, let us examine $C_{t+2i}$, $C_{t+2i+1}$, and $C_{t+2i+2}$, 
when we put 
$C_{t+2i}$ into the form as in (\ref{eq:form_of_D}). Note that $C_{t+2i+2}$ 
has 
been put into the desired form as in (\ref{eq:form_of_D}). That is, 
$${\small
C_{t+2i}=\begin{bmatrix}
\hat C_{t+2i} & \vstar & \vstar & \vzero & \vzero& \dots & \vzero \\
\vstar & 0 & 1 & 0 & 0 & \dots & 0 \\
\vstar & -1 & 0 & 0 & 0 & \dots & 0 \\
\vzero & 0 & 0 & 0 & 0 & \dots & 0 \\
\vzero & 0 & 0 & 0 & 0 & \dots & 0 \\
\vdots & \vdots & \vdots & \vdots &\vdots & \ddots & \vdots \\
\vzero & 0 & 0 & 0 & 0 &\dots & 0 
\end{bmatrix},}$$
$${\small
C_{t+2i+1}=\begin{bmatrix}
\hat C_{t+2i+1} & \vstar & \vstar &  \vstar & \vzero & \dots & \vzero \\
\vstar & 0 & * & * & 0 &\dots & 0 \\
\vstar & * & 0 & 1 & 0 & \dots & 0 \\
\vstar & * & -1 & 0 & 0 & \dots & 0 \\
\vzero & 0 & 0 & 0 & 0 & \dots & 0 \\
\vdots & \vdots & \vdots & \vdots &\vdots & \ddots & \vdots \\
\vzero & 0 & 0 & 0 & 0 &\dots & 0 
\end{bmatrix},}$$
$${\small 
C_{t+2i+2}=\begin{bmatrix}
\hat C_{t+2i+2} & \vstar & \vstar & \vzero & \vzero & \dots & \vzero \\
\vstar & 0 & * & 0 & 0 & \dots & 0 \\
\vstar & * & 0 & 0 & 0 &\dots & 0 \\
\vzero & 0 & 0 & 0 & 1 &\dots & 0 \\
\vzero & 0 & 0 & -1 & 0 &\dots & 0 \\
\vdots & \vdots & \vdots & \vdots & \vdots &\ddots & \vdots \\
\vzero & 0 & 0 & 0 & 0 &\dots & 0 
\end{bmatrix},}
$$
where $\hat C_{t+2i}$, $\hat C_{t+2i+1}$, and $\hat C_{t+2i+2}$ are in 
$\Lambda(s+2i-1, \F)$.
From the above, when we use $(t+2i, t+2i+1)$ and $(t+2i+1, t+2i)$ entries  
to set other entries on the $(t+2i)$th and $(t+2i+1)$th rows and columns in 
$C_{t+2i}$ to be zero, such operations
do not affect the $(t+2i+2)$th and $(t+2i+3)$th rows and columns of 
$C_{t+2i+2}$, nor the $(t+2j)$th and $(t+2j+1)$th rows and columns of $C_{t+2j}$ 
for $j\geq i+1$ in general. Furthermore, such operations do not change $C_{t+2j}$ 
for $j\leq i-1$ at all. It follows that after these operations, all $C_{t+2i}$, 
$i\in[\binom{t}{2}]$, are in the form of (\ref{eq:form_of_D}) as desired.

%Since the 
%$(s+2i, s+2i+1)$th entry of $B_{s+2i}$ is $1$, we can start from this entry, and 
%use row and column operations to set other entries in the $(s+2i)$th row and 
%column, and the $(s+2i+1)$th row and column, to be $0$. Since we start from 
%$i=\binom{s}{2}$ and decrease to $i=1$, such operations could only possibly 
%affect 
%$B_{s+2i-1}$. 

Suppose $(C_1, \dots, C_{t'})$ are changed to $(C_1', 
\dots, 
C_{t'}')$ after these operations, which implicitly define $Q\in\GL(t'+1, \F)$. We 
then do the following. For $i\in[t]$, let $D_i=C_i'$. For 
$i\in[\binom{t}{2}]$, let $D_{t+i}=C_{t+2i}'$. We then obtain 
$r=t+\binom{t}{2}=\binom{t+1}{2}$ matrices 
$D_1, \dots, D_r\in\Lambda(t'+1, \F)$ with $t'=t^2$ which satisfy the properties 
as required by this lemma, concluding the proof.
\end{proof}

%such that $D_{s+i}$ is of the 
%form 
%\begin{equation}\label{eq:form_of_D_2}
%\begin{bmatrix}
%\tilde D_{s+i} &  \vzero & \vzero & \vzero & \dots \\
%\vzero &  0 & 1 & 0 & \dots \\
%\vzero &  -1 & 0 & 0 & \dots \\
%\vzero &  0 & 0 & 0 & \dots \\
%\vdots  & \vdots & \vdots & \vdots & \ddots 
%\end{bmatrix},
%\end{equation}
%where $\tilde D_{s+i}$ is of size $s+2(i-1)$.

\subsection{Step 4} To start with, we need an observation on complete 
spaces as follows. 
%\subsection{Alternating matrix spaces and 3-way arrays.} 
%\subsubsection{From alternating bilinear maps to $3$-way arrays.} 
Given an $m$-dimensional $\cA\leq\Lambda(n, \F)$, let
$\bA=(A_1, 
\dots, A_m)\in \Lambda(n, \F)^m$ be an ordered basis of $\cA$. 
From $\bA$, we construct a $3$-way 
array $\tA\in\M(n\times n\times m, \F)$ whose frontal slices are $A_i$'s. Let 
$f_{i,j}\in \F^m$, $i, j\in[n]$, be the tube fibres of $\tA$. Let
$W=\span\{ e_1, \dots, e_t\}\leq \F^n$, where $e_i$'s are standard basis 
vectors. We then note the following characterisation of $W$ to be a complete space 
of $\cA$.
\begin{observation}\label{obs:tube}
Let $\cA$, $f_{i,j}$, and $W$ be as above. Then $W$ is a complete space of 
$\cA$, 
if and only 
if, $f_{i,j}$, 
$1\leq i<j\leq t$, are linearly independent. 
\end{observation}

The following lemma completes Step 4. 
\begin{lemma}\label{lem:step4}
Let $t'=t^2$, and $r=t+\binom{t}{2}$.
Suppose $\cD\leq \Lambda(t'+1, 
\F)$ contains 
$D_1, 
\dots, D_r\in \Lambda(t'+1, \F)$, such that (1) for $i\in [t]$, 
$D_i=\begin{bmatrix}
\tilde D_i & \vzero \\
\vzero & \vzero 
\end{bmatrix}$ where $\tilde D_i\in \Lambda(i+1, \F)$ and $D_i(i, i+1)=1$, 
and 
(2) for 
$i\in[\binom{t}{2}]$, $D_{t+i}$ is of the form
%$D_{r+i}=\begin{bmatrix}
%D_i' & \vzero & \vzero & \vzero \\
%\vzero & 0 & 1 & \vzero \\
%\vzero & -1 & 0 & \vzero \\
%\vzero & \vzero & \vzero & \vzero 
%\end{bmatrix}$, where $D_i'$ is of size $s+2(i-1)$.
$i\in[\binom{t}{2}]$, $D_{t+i}=\begin{bmatrix}
\tilde D_{t+i} & \vzero & \vzero & \vzero \\
\vzero & 0 & 1 & \vzero \\
\vzero & -1 & 0 & \vzero \\
\vzero & \vzero & \vzero & \vzero 
\end{bmatrix}$,
where $\tilde D_{t+i}\in\Lambda(t+1+2(i-1), \F)$. Then there exists $Q\in\GL(t'+1, 
\F)$, such that $W=\span\{ e_1, \dots, e_{t+1}\}$ is a complete space for 
$\tr{Q}\cD Q$. 
\end{lemma}

%Let $t=\binom{s+1}{2}$. From Step 2, we have $B_1, \dots, 
%B_t\in \Lambda(t+1, \F)$, such that for any $i\in[t]$, we have 
%$e_i^tB_ie_{i+1}\neq 0$, and $e_k^tB_ie_j=0$ for any $k\in [t+1]$ and any 
%$j>i+1$. 
% Let $D_1, \dots, D_r\in 
%\Lambda(t'+1, \F)$ be from Step 3. 
\begin{proof}
Construct a $3$-way array $\tD$ of size $(t'+1)\times(t'+1)\times r$, where the 
$i$th frontal slice is $D_i$. Let $f_{i,j}\in\F^r$ be the $(i, j)$th tube fibre of 
$\tD$. Note that the tube fibres $f_{1, 2}, f_{2,3}, \dots, f_{t, t+1}$ and 
$f_{t+2, t+3}, f_{t+4, t+5}, \dots, f_{t', t'+1}$ are linearly 
independent.

Let us arrange the tube fibres $f_{i,j}$, $1\leq i\leq t-1$, $i+2\leq j\leq t$, in 
the reverse lexicographic order, and relabel them accordingly as $\tilde f_k$ for 
$k\in[\binom{t}{2}]$. That is, $\tilde f_1=f_{1,3}$, $\tilde f_2=f_{1,4}$, $\tilde 
f_3=f_{2, 4}$, $\tilde f_4=f_{1, 5}$, $\tilde f_5=f_{2,5}$, and so on.

%Our goal is to use row and column operations to make the tube fibres 
%$f_{k, \ell}$, $1\leq k<\ell\leq t+1$, to be linearly independent. 
Our goal is to apply appropriate elementary matrices (on the left and their 
transposes on the right) to $\tD$ to make the tube 
fibres at positions $(k, \ell)$, $1\leq k<\ell\leq t+1$, linearly independent.
If this could 
be achieved, Observation~\ref{obs:tube} ensures that $W=\span\{ e_1, \dots,
e_{t+1}\}$ is a complete space of the resulting alternating matrix space.

To do that, consider the following operations in $\binom{t}{2}$ rounds. After the 
$i$th round, we wish to maintain that 
$$f_{1, 2}, f_{2,3}, \dots, f_{t, t+1}, \tilde f_1, 
\dots, \tilde f_i, f_{t+2(i+1), t+2i+3}, f_{t+2(i+2), t+2i+5}, \dots, f_{t', 
t'+1}$$
are linearly 
independent. Note that $t'=t+2\cdot \binom{t}{2}$. If this could be achieved, 
after $\binom{t}{2}$ rounds, 
$$f_{1,2}, f_{2,3}, \dots, f_{t, t+1}, \tilde f_1, \dots, \tilde 
f_{\binom{t}{2}}$$ 
would be linearly independent.

Recall that before the first round starts, we have that the tube fibres $f_{1, 
2}$, 
$f_{2,3}$, $\dots$, $f_{t, t+1}$ and 
$f_{t+2, t+3}, f_{t+4, t+5}, \dots, f_{t', t'+1}$ are linearly 
independent. 

Suppose now we have completed the $i$th round. Let us explain the operations in 
the 
$(i+1)$th round. We first check if 
$$f_{1, 
2}, f_{2,3}, \dots, f_{t, t+1}, \tilde f_1, 
\dots, \tilde f_i, \tilde f_{i+1}, f_{s+2(i+2), s+2i+5}, 
f_{s+2(i+3), s+2i+7}, \dots, f_{t', 
t'+1},$$ 
are linearly independent. 

If so, we proceed to the next round. 

If not, we have that $\tilde f_{i+1}$ is in the linear span of 
$$f_{1, 
2}, f_{2,3}, \dots, f_{t, t+1}, \tilde f_1, 
\dots, \tilde f_i, f_{t+2(i+2), t+2i+5}, 
f_{t+2(i+3), t+2i+7}, \dots, f_{t', 
t'+1}.$$ 
Now we 
wish to add $f_{t+2(i+1), t+2i+3}$ to 
$\tilde f_{i+1}$. This is because, since after the $i$th round we had that
\begin{equation}\label{eq:cond}
f_{1, 2}, f_{2,3}, 
\dots, f_{t, t+1}, \tilde f_1, 
\dots, \tilde f_i, f_{t+2(i+1), t+2i+3}, f_{t+2(i+2), t+2i+5}, \dots, f_{t', 
t'+1},
\end{equation}
are linearly independent, we have that 
\begin{equation}\label{eq:desired}
f_{1, 2}, f_{2,3}, 
\dots, f_{t, t+1}, \tilde f_1, 
\dots, \tilde f_i, \tilde f_{i+1}+f_{t+2(i+1), t+2i+3}, f_{t+2(i+2), t+2i+5}, 
\dots, f_{t', 
t'+1},
\end{equation}
are also linearly independent. 

But we cannot add $f_{t+2(i+1), t+2i+3}$ to 
$\tilde f_{i+1}$ directly. Indeed, the legitimate operations are left multiplying  
elementary matrices (namely $I+E_{i,j}$) and right multiplying 
their transposes. 
These 
correspond to performing row and column operations on $\tD$ viewed as a matrix 
$(f_{i,j})_{i,j\in[t'+1]}$ whose entries are vectors. We will make use of this 
perspective in the following. 

Suppose $\tilde f_{i+1}$ corresponds to 
$f_{j, k}$ for 
some $1\leq j<k\leq t+1$. In order to add $f_{t+2(i+1), t+2i+3}$ to 
$\tilde f_{i+1}$, we can first add the $(t+2(i+1))$th row to the $j$th row, and to 
maintain the alternating property, add 
the $t+2(i+1)$th column to the $j$th column. Then we add the $(t+(2i+3))$th column 
to the $k$th column, and to maintain the alternating property, add the 
$(t+(2i+3))$th row to the $k$th row. This does add $f_{t+2(i+1), t+2i+3}$ to 
$\tilde f_{i+1}$. 

However, it is 
possible that during the above procedure, some of the vectors in 
$$\{f_{1, 
2}, f_{2,3}, \dots, f_{t, t+1}, \tilde f_1, 
\dots, \tilde f_i\}$$ get altered as well. 
(It is easy to see that $f_{t+2(i+2), t+2i+5}, \dots, f_{t', t'+1}$ are not 
changed.)
For example, when adding the $(t+2(i+1))$th row to the $j$th row, $f_{j, j+1}$ and 
those $\tilde f_{i'}$ corresponding to $f_{j, j'}$ could be added by certain 
vectors as 
well. 
Therefore, instead of getting those vectors in (\ref{eq:desired}), we 
 get
\begin{multline}\label{eq:actual}
 f_{1, 2}+g_{1,2}, f_{2,3}+g_{2,3}, 
\dots, f_{t, t+1}+g_{t,t+1}, \tilde f_1+\tilde g_1, 
\dots, \tilde f_i+\tilde g_i, \\
 \tilde f_{i+1}+f_{t+2(i+1), t+2i+3}+\tilde g_{i+1}, f_{t+2(i+2), t+2i+5}, \dots, 
 f_{t', t'+1},
\end{multline}
where $g_{j,j+1}$ and $\tilde g_k\in \F^r$. 

Therefore, we need to show that those 
vectors in (\ref{eq:actual}) are linearly independent. The following 
observation is crucial for this. 
\begin{observation}\label{obs:key}
We have that $g_{j, j+1}$ and $\tilde 
g_k$ are in $\span\{f_{t+2(i+2), t+2i+5}$, $\dots$, $f_{t', t'+1}\}$. 
\end{observation}
\begin{proof}
Note that $g_{j, j+1}$ and $\tilde g_k$ come from those fibre tubes $f_{p, 
t+2(i+1)}$ and $f_{q, t+2i+3}$, where $1\leq p, q\leq t+1$. As can be seen from 
(\ref{eq:form_of_D}), 
these fibre tubes are in the linear span of $f_{t+2(i+2), t+2i+5}, \dots, f_{t', 
t'+1}$, because the only non-zero entries on the $(t+2i+2)$th and $(t+2i+3)$th 
columns of $D_{t+i+1}$ are in the $(t+2i+3, t+2i+2)$ and $(t+2i+2, t+2i+3)$th 
positions.
\end{proof}

Given 
this observation, the linear independence of vectors in (\ref{eq:actual})
follows from the linear independence of vectors in (\ref{eq:cond}). Indeed, 
by a change of basis, we can assume that the vectors in (\ref{eq:cond}) 
form a set of standard basis vectors in the order they are listed. So putting 
those vectors in (\ref{eq:cond}) as 
column vectors in  
a matrix form simply gives
$$
\begin{bmatrix}
I_{t+i} & \vzero &\vzero \\
\vzero & 1 & \vzero \\
\vzero & \vzero & I_{\binom{t+1}{2}-(t+i+1)}
\end{bmatrix},
$$
where $I_k$ denotes the $k\times k$ identity 
matrix. Now putting the vectors in (\ref{eq:actual}) as column vectors in a 
matrix gives
$$
\begin{bmatrix}
I_{t+i} & {\bf *} &\vzero \\
\vzero & 1 & \vzero \\
\mathbf{*} & * & I_{\binom{t+1}{2}-(t+i+1)}
\end{bmatrix},
$$
where $\mathbf{*}$ means that the entries could be arbitrary. This matrix is 
clearly of full-rank, proving the linear independence of vectors in 
(\ref{eq:actual}). Note that the 
entries in the lower-left submatrix comes from $g_{i, i+1}$ and $\tilde g_j$, and 
the entries in the $(t+i+1)$-th column comes from $f_{t+2(i+1), t+(2i+3)}+\tilde 
g_{i+1}$. This 
also explains the necessity of Step 3, as otherwise the 
upper-left $(t+i+1)\times(t+i+1)$ submatrix would be of the form
$\begin{bmatrix}
I_{t+i} & {\bf *}\\
{\bf *} & 1
\end{bmatrix}$, which could be not full-rank.

Now that we achieved what we wanted in the $(i+1)$th round, also note that the 
above operations do not 
affect the $2j$ and $2j+1$ rows and columns of $D_{t+j}$ for $j>i+1$. This means 
that we can have the same set up to perform the above operations (with increased 
row and column indices) in the next round. This concludes the proof of 
Lemma~\ref{lem:step4}.
\end{proof}

This 
concludes the proof of Theorem~\ref{thm:main}.\hfill\qed

\section{Applications of Theorem~\ref{thm:main}}\label{sec:app}

To explain the applications of Theorem~\ref{thm:main} in group theory and 
geometry, we recast Theorem~\ref{thm:main} 
in terms of alternating bilinear maps, which follows easily by the relationship 
between alternating matrix spaces and alternating bilinear maps explained in 
Section~\ref{subsec:rel}.

Let $\phi:\F^n\times\F^n\to\F^m$ be an alternating bilinear map. Let 
$W\leq \F^n$ be a subspace. We say that $W$ is a \emph{totally-isotropic} space 
for $\phi$, if 
$\phi(W, W)=0$. We say that $W$ is a \emph{complete space} for $\phi$, if 
$\dim(\span(\phi(W, W)))=\binom{\dim(W)}{2}$. 
\begin{theorem}[Theorem~\ref{thm:main} for alternating bilinear 
maps]\label{thm:main2}
Let $s, t\in \N$, $s, t\geq 2$. Let $\phi: \F^n\times \F^n\to \F^m$ be an 
alternating 
bilinear 
map over $\F$, where 
$n\geq s\cdot t^4$. Then $\phi$ has either a dimension-$s$ totally-isotropic 
space, or a dimension-$t$ complete space. 
\end{theorem}

\subsection{Implications and new questions in group 
theory}\label{subsec:implications}

%\paragraph{Implications.} 
Let $p$ be an odd prime. Let $\fB_{p,2}$ be the class of 
finite $p$-groups of 
class $2$ and exponent $p$. That is, a finite group $G$ is in $\fB_{p,2}$, if the 
commutator subgroup $[G, G]$ is contained in the centre 
$\mathrm{Z}(G)$, and every $g\in G$ satisfies that $g^p=\id$. We also define 
$\fB_{p,2,n}\subseteq \fB_{p,2}$, such that $G\in \fB_{p,2}$ is in $\fB_{p,2,n}$ 
if and only if a minimal generating set of $G$ is of size $n$, or equivalently, if 
$G/[G,G]\cong \Z_p^n$. 

There are two important group families in $\fB_{p,2}$. First, elementary abelian 
$p$-groups, $\Z_p^s$, are in $\fB_{p,2}$. Second, for any $t\in \N$, 
there 
is $F_{p,2,t}$, the \emph{relatively free} $p$-groups of class $2$ and exponent 
$p$ with 
$t$ generators, defined as the quotient of the 
free group in $t$ generators by the subgroup generated by all words of the form 
$x^p$ and $[[x, y], z]$. Note that $F_{p,2,t}$ can be viewed as a universal group 
in 
$\fB_{p,2,t}$, in that any group in $\fB_{p,2,t}$ is isomorphic to the
quotient of $F_{p,2,t}$ by a subgroup of $[F_{p,2,t}, F_{p,2,t}]$.

Baer's correspondence \cite{Bae38} connects $\fB_{p,2}$ 
with alternating bilinear maps over 
$\F_p$. Indeed, this correspondence leads to an isomorphism between the categories 
of groups in $\fB_{p,2}$ and of alternating bilinear maps over $\F_p$ (cf. 
\cite[Sec. 3]{Wil09a}). It is then not surprising to see correspondences between 
structures of $\fB_{p,2}$ and of alternating bilinear maps over 
$\F_p$. Examples include abelian subgroups vs totally-isotropic spaces 
\cite{Alp65}, central decompositions vs orthogonal decompositions 
\cite{Wil09a,LQ19}, 
hyperbolic pairs vs totally-isotropic decompositions \cite{BMW17,BCG+19}. 
%
%Interpreting Theorem~\ref{thm:main} in the context of Baer's correspondence, we 
%have the following corollary. For completeness we include a proof in
%Section~\ref{subsec:proof_group}.

%By known connections between $p$-groups and alternating bilinear maps which dates 
%back to Baer's work in the 1930s \cite{Bae38}, 
Theorem~\ref{thm:main2} has a natural interpretation in the context of $p$-groups 
of class $2$ and exponent $p$ as follows. 
\begin{corollary}\label{cor:groups}
Let $G\in\fB_{p,2,n}$, where $n\geq s\cdot t^4$. Then $G$ has either an abelian 
subgroup 
$S\leq G$ such that $S[G,G]/[G,G]\cong \Z_p^s$, or a subgroup isomorphic 
to $F_{p,2,t}$.
%Let $p$ be an odd prime. Let $G$ be a $p$-group of class $2$ and exponent $p$, 
%such that its commutator quotient $G/[G,G]\cong \Z_p^{n}$, where $n\geq s\cdot 
%t^4$. Then $G$ has either an abelian 
%subgroup 
%$S\leq G$ such that $S[G,G]/[G,G]\cong \Z_p^s$, or a subgroup isomorphic 
%to the relatively free $p$-group of class $2$ and exponent $p$ with $t$ 
%generators. 
%$F_{p,2,t}$.
\end{corollary}
\begin{proof}
Let $G\in \fB_{p,2,n}$. Then $G/[G, G]\cong \Z_p^n$ and suppose $[G, G]\cong 
\Z_p^m$. The commutator map $[, ]$ induces an alternating 
bilinear map $\phi: G/[G,G]\times G/[G,G]\to [G,G]$. 

Given a subgroup $H\leq G/[G, G]$ such that $H\cong\Z_p^s$, let $S_H$ be a 
subgroup of $G$ of the smallest order satisfying $S_H[G,G]/[G,G]=H$. Then it is 
known, at least since \cite{Alp65}, that $H$ is a totally-isotropic space of 
$\phi$ if and only if $S_H$ is abelian. It is also straightforward to verify 
that $H$ is a complete space if and only if $S_H$ is isomorphic to $F_{p,2,s}$. 
From these, the corollary follows immediately from Theorem~\ref{thm:main2}.
\end{proof}

Readers familiar with varieties of groups \cite{Neu67} may recognise that abelian 
groups and relatively free groups are the two opposite structures in a variety of 
groups. In this sense, Corollary~\ref{cor:groups} may be viewed as a 
group-theoretic version of the classical Ramsey theorem for graphs \cite{Ram30}.

Corollary~\ref{cor:groups} also leads to the following family of 
questions. 
Recall that a 
variety of groups, $\fC$, is the class of all groups satisfying a set of laws.
Examples 
include abelian groups, nilpotent groups of class $c$, and solvable groups of 
class $c$.  
Let $\fC_t$ be the subclass of $\fC$, consisting of groups that
can be generated by $t$ elements. The relatively free 
group of rank $t$ in $\fC_t$, $F_{\fC, t}$, as the free group on $t$ generators 
modulo the 
laws defining $\fC$. For a group $G$ and $S\subseteq G$, let $\Phi(G)$ be the 
Frattini subgroup of $G$. Then the following Ramsey problem for $\fC$ can be 
formulated.
\begin{question}[Ramsey problem for a variety of groups $\fC$]\label{q:groups}
Let $G\in \fC_n$ and $s, t\in \N$. Is it true that, if $n>f_\fC(s, t)$ for some 
function 
$f_\fC:\N\times \N\to\N$, then 
either 
there exists an abelian subgroup $S\leq G$ such that $S\Phi(G)/\Phi(G)$ is of 
rank 
$s$, or $G$ has a subgroup isomorphic to $F_{\fC, t}$.
\end{question}
There are several deep Ramsey-type results 
for 
nilpotent groups \cite{Lei98,BL03,JR17}, mostly following the lines of the van der 
Waerden theorem \cite{Wae27} and the Hales-Jewett theorem \cite{HJ63}. 
Question~\ref{q:groups} asks to show the existence of large enough 
subgroups of certain types in a 
group from 
a variety of groups \cite{Neu67}, which is in a 
closer analogy with the graph Ramsey theory.

\subsection{Implications and questions for Grassmannians}\label{subsec:grass}
Theorem~\ref{thm:main2} 
can also be interpreted in the context of hyperplane sections of Grassmannians. To 
introduce this implication we need some further terminologies. 

Let $\phi:(\F^n)^{\times \ell}\to\F^m$ be an alternating 
$\ell$-linear map. We say 
that $V\leq\F^n$ is a \emph{totally-isotropic space}
of $\phi$, if for any $v_1, 
\dots, v_\ell\in V$, $\phi(v_1, \dots, v_\ell)=0$. 
We say that $W\leq\F^n$ of dimension $\geq \ell$ is an \emph{anisotropic space} of 
$\phi$, if for any linearly independent $w_1, \dots, 
w_\ell\in W$, $\phi(w_1, \dots, w_\ell)\neq 0$. 
%If $W\leq \F^n$ of dimension $\geq 
%\ell$ satisfies this 
%condition, we say that $W$ is an \emph{anisotropic space} of $\phi$. 

As pointed out by Feldman and Propp \cite[Sec. 6]{FP92}, an alternating 
$\ell$-linear map $\phi:(\F^n)^{\ell}\to\F^m$ defines an $m$-fold hyperplane 
section $H$ on the 
Grassmannian $\Gr(n, \ell)$, the variety of $\ell$-dimensional subspaces of 
$\F^n$. For $W\leq \F^n$, $\Gr(W, \ell)$ is a subvariety of $\Gr(n, \ell)$. 
Feldman and Propp noted that $W$ is a totally-isotropic space if and only if 
$\Gr(W, \ell)$ is contained in $H$. It is also easy to see that $W$ is anisotropic 
if and only if $H$ intersects $\Gr(W, \ell)$ trivially. 

Let $\ell=2$, and note that a complete space is anisotropic. Due to the geometric 
interpretations of anisotropic spaces and totally-isotropic spaces explained 
above, 
Theorem~\ref{thm:main2} implies the following. 

%By $\ARN_\F(s, t, 2)\leq \CRN_\F(s, t, 2)$, 
%we  
%have $\ARN_\F(s, t, 2)\leq s\cdot t^4$. Due to the geometric interpretations of 
%complete spaces and totally-isotropic spaces explained above, 
%Theorem~\ref{thm:main} has the following implication for grassmannians. 

\begin{corollary}\label{cor:geometry}
When $n\geq s\cdot t^4$, any $m$-fold hyperplane section of $\Gr(n, 2)$ either 
contains $\Gr(W, 2)$ for some $s$-dimensional $W\leq \F^n$, or intersects $\Gr(W, 
2)$ trivially 
for some $t$-dimensional $W\leq \F^n$. 
\end{corollary}
Note that Corollary~\ref{cor:geometry} deals with $m$-fold hyperplane sections of 
$\Gr(n, 2)$. This leads to the question as whether a similar statement for $\Gr(n, 
\ell)$, $\ell>2$, holds. 
%Corollary~\ref{cor:geometry} naturally leads to ask the question for $\Gr(n, 
%\ell)$ for $\ell>2$.
% as already mentioned in Section~\ref{subsec:new_q}.

%Proposition~\ref{prop:main} follows easily
%from the results of Ol'shanskii (\cite[Lemma 2]{Ols78}) and Buhler, Gupta, and 
%Harris (\cite[Main Theorem]{BGH87}), which were obtained in 
%the context of abelian subgroups of $p$-groups 
%\cite{Bur13,Alp65}. 
%We remark again that the proofs in \cite{Ols78,BGH87} 
%are based on probabilistic methods in the linear algebraic or geometric settings, 
%which 
%echoes Erd\H{o}s' use of the probabilistic method to show a lower bound for graph 
%Ramsey numbers \cite{Erd47}.

\section{Discussions: Theorem~\ref{thm:main} as a linear Ramsey 
theorem}\label{sec:discussion}

In this section we discuss on an analogy between graphs and alternating bilinear 
maps, and more generally, between hypergraphs and alternating multilinear maps, in 
the context of extremal combinatorics. 
%Some works of particular relevance in this 
%discussions are by Lov\'asz \cite{Lov77}, Buhler, Gupta, and Harris \cite{BGH87}, 
%and Feldman and Propp \cite{FP92}.

Through this analogy, Theorem~\ref{thm:main} can be naturally understood as a 
linear algebraic Ramsey theorem. A previous ``linear Ramsey theorem'' of Feldman 
and 
Propp 
\cite{FP92} could be more properly understood as a ``linear Tur\'an theorem.'' An 
intriguing open problem, which can be viewed as a linear Tur\'an's problem for 
hypergraphs, is proposed in Section~\ref{subsec:turan_problem}. 

We also present two small results. First, we present a result that 
complements  
Theorem~\ref{thm:main} (Proposition~\ref{prop:main}). Second, we generalise a 
correspondence between independent sets and totally-isotropic spaces in 
\cite{BCG+19} from graphs to hypergraphs (Proposition~\ref{prop:alpha}).

\subsection{Graphs and alternating matrix spaces} Following ideas traced back to 
Tutte \cite{Tut47} and Lov\'asz \cite{Lov79}, we construct an alternating 
matrix space from a graph.

Let $G=([n], E)$ be a simple, undirected graph. Suppose 
$E=\{\{i_1, j_1\}, \dots$, $\{i_m, j_m\}\}\subseteq\binom{[n]}{2}$, where for 
$k\in[m]$, $i_k<j_k$ and for 
$1\leq k<k'\leq m$, $(i_k, j_k)<(i_{k'}, j_{k'})$ in the 
lexicographic order. For $k\in[m]$, let $A_k$ be the
alternating matrix 
$A_{i_k, j_k}$ (defined in Section~\ref{sec:prel}) over $\F$, and 
set $\cA_G:=\span\{A_1, \dots, A_m\}\leq\Lambda(n, \F)$. 
%
%set 
%$\bA_G=(A_1, \dots, A_m)\in\Lambda(n, \F)^m$. Then $\bA_G$ defines an alternating 
%bilinear map 
%$\phi_G:\F^n\times\F^n\to\F^m$, by $\phi_G(v, u)=\tr{(\tr{v}A_1u, \dots, 
%\tr{v}A_mu)}$.

%So in the following, we may transfer results stated in one 
%setting to another without explicitly stating the origins.

The basic observation of Tutte and Lov\'asz is that $G$ has a perfect matching if 
and only if $\cA_G$ contains a full-rank matrix. 
%there exists a codimension-$1$ subspace $W\leq \F^m$, such that the 
%bilinear form obtained by composing $\phi_G$ with quotienting out $W$ is 
%non-degenerate\footnote{Tutte and Lov\'asz actually examined the linear space 
%spanned 
%by $A_i$'s, namely $\cA=\langle A_1, \dots, A_m\rangle\leq\Lambda(n, \F)$. It is 
%easy to see that alternating matrix spaces and alternating bilinear maps are 
%closely related. }. 
This classical example is the precursor of several recent 
discoveries relating properties of $G$, including 
independent sets and vertex colourings \cite{BCG+19}, vertex and edge 
connectivities \cite{LQ19}, and isomorphism and homomorphism notions \cite{HQ20a}, 
with properties of $\cA_G$. 
Graph-theoretic questions and techniques have also been translated to study 
alternating matrix spaces or alternating bilinear maps in \cite{LQ17,BCG+19,Qia20} 
with 
applications to group theory and quantum information.

\subsection{Hypergraphs and alternating multilinear 
maps (and exterior algebras)}\label{subsec:hypergraph_alternating}
Naturally extending the 
construction of alternating matrix spaces from graphs, the following classical 
construction 
of
alternating $\ell$-linear maps from $\ell$-uniform hypergraphs is also traced 
back  to Lov\'asz \cite{Lov77}.

Given
$\{i_1, \dots, i_\ell\}\in \binom{[n]}{\ell}$ where $i_1<\dots<i_\ell$, there 
is the
alternating 
$\ell$-linear form $e_{i_1}^*\wedge \dots \wedge e_{i_\ell}^*$. From an 
$\ell$-uniform hypergraph $H=([n], E)$ where $E\subseteq\binom{[n]}{\ell}$ and 
$|E|=m$, we can construct an alternating $\ell$-linear map $\phi_H:(\F^n)^{\times 
\ell}\to\F^m$ by first applying the above alternating $\ell$-linear form 
construction to 
every hyperedge in $E$, and then ordering these forms by 
the lexicographic
ordering of $\binom{[n]}{\ell}$. 

Lov\'asz's construction was originally stated in 
terms of subspaces of exterior algebras. Since then, the use of exterior 
algebras has lead to 
the elegant extensions of Bollob\'as's Two Families Theorem \cite{Bol65} by Frankl 
\cite{Fra82}, Kalai \cite{Kal84} and Alon \cite{Alo85}, as well as Kalai's 
algebraic shifting method \cite{Kal02}. The recent work by Scott and Wilmer 
\cite{SW19} systematically extends several basic results from the extremal 
combinatorics of hypergraphs to subspaces of exterior algebras.

Subspaces of exterior algebras, linear spaces of alternating multilinear forms, 
and alternating 
multilinear maps are of course closely related. Indeed, our main result can be 
formulated in terms of exterior algebras, just as Feldman and Propp did for their 
main result in \cite[Corollary 2]{FP92}.

\subsection{Independent sets and totally-isotropic spaces} 
\label{subsec:iso_number}
Recall that for an 
$\ell$-uniform hypergraph $H=([n], E)$, $S\subseteq [n]$ is an \emph{independent 
set} in $H$, if 
$S$ does not contain any hyperedge from $E$. The \emph{independence number} of 
$H$, 
$\alpha(H)$, is the maximum size over all independent sets in $H$. 

Let $\phi:(\F^n)^{\times \ell}\to\F^m$ be an alternating $\ell$-linear map. 
%Then 
%$W\leq\F^n$ is a \emph{totally-isotropic space} of $\phi$, if for any $w_1, 
%\dots, 
%w_\ell\in 
%W$, 
%$\phi(w_1, \dots, w_\ell)=0$. 
Recall that totally-isotropic spaces of $\phi$ are defined in 
Section~\ref{subsec:grass}.
The 
\emph{totally-isotropic 
number} of 
$\phi$, $\alpha(\phi)$, is the maximum dimension over all totally-isotropic spaces 
of 
$\phi$. 

In \cite{BCG+19}, it is shown that when $\ell=2$, i.e. for a graph $G$, 
$\alpha(G)=\alpha(\phi_G)$. We generalise that result to any $\ell$ in the 
following proposition, which justifies viewing totally-isotropic spaces as a 
linear 
algebraic analogue of independent sets. 
%, whose proof is in Section~\ref{subsec:proof_alpha}.
\begin{proposition}\label{prop:alpha}
Let $H$ be an $\ell$-uniform hypergraph, and let $\phi_H$ be the alternating 
$\ell$-linear map constructed from $H$ as in 
Section~\ref{subsec:hypergraph_alternating}. Then we have 
$\alpha(H)=\alpha(\phi_H)$.
\end{proposition}
\begin{proof}
Let $H=([n], E)$, $E\subseteq\binom{[n]}{\ell}$. Let 
$\phi_H:(\F^n)^{\times\ell}\to\F^m$ be the alternating $\ell$-linear map 
constructed from $H$ via the construction in 
Section~\ref{subsec:hypergraph_alternating}. 

Suppose $S\subseteq[n]$ is an 
independent set of $H$. Let $V=\span\{ e_i \mid i\in S\}\leq \F^n$. It is 
easy to 
verify that $V$ is a totally-isotropic space of $\phi_H$. It follows that 
$\alpha(H)\geq \alpha(\phi_H)$.

Suppose $V\leq \F^n$ is a totally-isotropic space of $\phi_H$ of dimension $d$. 
Let $B\in\M(n\times d, \F)$ be a matrix whose columns span $V$. For 
$i\in[n]$, let 
$w_i\in \F^d$ such that $\tr{w_i}$ is the $i$th row of $B$. As $\rk(B)=d$, there 
exists $i_1, \dots, i_d$, $1\leq i_1<\dots<i_d\leq n$, such that $w_{i_j}$, 
$j\in[d]$, are linearly independent. 

We claim that $\{i_1, \dots, i_d\}\in[n]$ is 
an independent set of $H$. If not, by relabelling the vertices, we can assume that 
$\{i_1, \dots, i_\ell\}\in E$. As $V$ is a totally-isotropic space, we have 
$e_{i_1}^*\wedge \dots \wedge e_{i_\ell}^*$, when restricted to $V$, is the zero 
map. It follows that $w_{i_1}\wedge \dots \wedge w_{i_\ell}=0$, contradicting that 
$w_{i_1}, \dots, w_{i_\ell}$ are linearly independent. The claim is proved. 

We then derive that $\alpha(\phi_H)\geq \alpha(H)$, concluding the proof. 
\end{proof}

%\section{Tur\'an's problem }\label{sec:turan}

\subsection{Three numbers related to Tur\'an's theorem}\label{subsec:three}

For an $\ell$-uniform hypergraph $H=([n], E)$, 
the size of $H$ is denoted as $\size(H)=|E|$. We define three closely related 
numbers regarding 
$n$, $\size(H)$, $\ell$, and the independence number $\alpha(H)$. The first number 
is just 
the celebrated Tur\'an number for hypergraphs \cite{Tur61,Sid95}. The other two 
numbers originate from \cite{BGH87} and \cite{FP92}, respectively, in the context 
of alternating multilinear maps. 

Let $n, m, \ell, a\in \N$. 
The \emph{Tur\'an 
number} is 
$$\TN(n, a, \ell)=\min\{\size(H) \mid H \text{ is }\ell\text{-uniform, } 
n\text{-vertex hypergraph, and }\alpha(H)\leq a\}.$$

The \emph{Feldman-Propp number} is 
$$
\FP(a, m, \ell)=\min\{n\in \N \mid \forall \ell\text{-uniform, } 
n\text{-vertex, 
} m\text{-edge hypergraph }H, \alpha(H) > a\}.
$$

We also define the number 
$$\alpha(n, m, \ell)=\min\{\alpha(H) 
\mid  H \text{ is }\ell\text{-uniform, } 
n\text{-vertex, 
} m\text{-edge hypergraph}\}.$$ 
%It is closely 
%related to the Tur\'an number \cite{Tur61, Sid95} and a number defined by Feldman 
%and Propp in \cite{FP92}. For the readers' convenience, we explicitly define 
%these 
%numbers and explain their relations. 
%, and $\alpha(H)$ is the independence 
%number of 
%$H$.

It is easy to see the relations of these three numbers, $\alpha(n, m, \ell)$, 
$\TN(n, a, \ell)$, and 
$\FP(a, m, \ell)$. On the one hand, $\TN(n, a, \ell)\leq 
m$ implies the existence of some $n$-vertex, $m$-edge, $\ell$-uniform hypergraph 
$H$ with $\alpha(H)\leq a$, which in turn implies that $\alpha(n, m, \ell)\leq a$ 
and $\FP(a, m, \ell)>n$. On the other hand, $\TN(n, a, \ell)>m$ implies that for 
any 
$n$-vertex, $m$-edge, $\ell$-uniform hypergraph $H$, $\alpha(H)>a$, which in turn 
implies that $\alpha(n, m, \ell)>a$ and $\FP(a, m, \ell)\leq n$. 

In the alternating multilinear map setting, by replacing independence numbers with 
totally-isotropic numbers defined in Section~\ref{subsec:iso_number}, we can 
define $\alpha_\F(n, m, \ell)$, $\TN_\F(n, a, \ell)$, and 
$\FP_\F(a, m, \ell)$ for those alternating multilinear maps $\phi:(\F^n)^{\times 
\ell}\to\F^m$ with $\alpha(\phi)=a$. 
Note that an immediate consequence of Proposition~\ref{prop:alpha} is that 
$\alpha_\F(n, m, 
\ell)\leq \alpha(n, m, \ell)$.

\subsection{Tur\'an meets Buhler, Gupta, and Harris} \label{subsec:turan_bgh}
The celebrated Tur\'an's 
theorem \cite{Tur41} is a cornerstone 
of extremal graph theory \cite{Bol04}. 
Formulated in terms of independent sets, 
%this result can be stated as 
%follows. Let 
%$\alpha(n, m, \ell)$ be $\min\{\alpha(H) \mid  H \text{ is }\ell\text{-uniform, } 
%n\text{-vertex, 
%} m\text{-edge hypergraph}\}$. 
 Tur\'an's theorem gives that
\begin{equation}\label{eq:turan}
\alpha(n, m, 2)\geq \Big\lceil
\frac{n^2}{2m+n} \Big\rceil,
\end{equation}
where the equality can be achieved.
%
%Then we have
%\begin{equation}\label{eq:turan}
%\alpha(n, m)= 
%\Big\lceil \frac{n^2}{2m+n}\Big\rceil.
%\end{equation}
%and the inequality can be achieved. 
%For
%$\cA\leq \Lambda(n, \F)$, we define the \emph{isotropic number} of $\cA$ to be the
%dimension of a maximum isotropic space of $\cA$.

In the alternating multilinear map setting, 
%it is natural to define
%$\alpha_\F(n,
%m, \ell):=\min\{\alpha(\phi) \mid  \phi:(\F^n)^{\times \ell}\to\F^m \text{ is 
%alternating}\}$. 
%The 
the quantity $\alpha_\F(n, m, 2)$
has
been studied by Buhler, Gupta, and Harris \cite{BGH87}. The main result of 
\cite{BGH87} states that
% in relation to abelian
%subgroups of $p$-groups
%\cite{Bur13,Alp65}.
%if $\F$ is
%not of characteristic $2$ and $m>1$, then
for any $m>1$, we have 
\begin{equation}\label{eq:bgh}
\alpha_\F(n, m, 2)\leq \Big\lfloor \frac{m+2n}{m+2}\Big\rfloor,
\end{equation}
where the equality is attainable over any\footnote{While 
in \cite{BGH87} the main result was stated for fields of characteristic $\neq 2$, 
the 
proof works for any characteristic.} algebraically closed\footnote{For fields 
that are not  
algebraically closed, the equality may not be achieved. See \cite[Sec. 3]{BGH87} 
and \cite{GQ06}.} field. The inequality was also 
obtained earlier by Ol'shanskii \cite{Ols78} over finite fields. This 
allows us to show the following result that complements Theorem~\ref{thm:main}.

\begin{proposition}\label{prop:main}
There exists an alternating bilinear map $\phi: \F^n\times \F^n\to \F^m$, 
$n=\Theta(s\cdot 
t^2)$, 
such that $\phi$ has neither a dimension-$s$ totally-isotropic space, nor a 
dimension-$t$ complete space. 
\end{proposition}
\begin{proof}
Let $m=\lfloor \binom{t-1}{2}\rfloor$. Let 
$n=\lfloor\frac{(m+2)(s-2)}{2}\rfloor+1$, which implies that $\frac{m+2n}{m+2}\leq 
s-1$. Note that $n= \Theta(s\cdot m)= \Theta(s\cdot t^2)$.

By Ol'shanskii (\cite[Lemma 2]{Ols78}) and 
Buhler, Gupta, and 
Harris (\cite[Main Theorem]{BGH87}), $\alpha_\F(n, m, 2)\leq \lfloor 
\frac{m+2n}{m+2}\rfloor$. It follows that there exists 
$\phi:(\F^n)^{\times\ell}\to\F^m$ with $\alpha(\phi)\leq s-1$. Furthermore, by 
$m=\lfloor \binom{t-1}{2} \rfloor$, any complete spaces of $\phi$ is of dimension 
$\leq 
t-1$. The result then follows.
\end{proof}

%Theorem~\ref{thm:main} and Proposition~\ref{prop:main} together imply that 
%$\Omega(s\cdot t^2)\leq \CRN_\F(s, t, 2)\leq s\cdot t^4$. 
Closing the gap 
between 
$\Omega(s\cdot t^2)$ from Proposition~\ref{prop:main}, and $s\cdot t^4$ from 
Theorem~\ref{thm:main}, is an interesting open problem.
%(We use
%$\lfloor
%x \rfloor$ to denote the largest integer $\leq x$.) 

\begin{remark}\label{rem:compare}
Comparing Equations~\ref{eq:turan} and~\ref{eq:bgh}, we see that $\alpha(n, m, 2)$
and $\alpha_\F(n, m, 2)$ behave quite differently. For example, by
Equation~\ref{eq:turan}, every graph with $n$ vertices and $2n$ edges
has an independent set of size at least $n/5$. On the other hand, by
Equation~\ref{eq:bgh}, there exists an alternating bilinear map 
$\phi:\F^n\times\F^n\to\F^{2n}$ with no 
totally-isotropic spaces of dimension $\geq 2$. 
\end{remark}

The results of \cite{Ols78,BGH87} were obtained in the context of abelian 
subgroups of finite $p$-groups, following the works of Burnside \cite{Bur13} and 
Alperin \cite{Alp65}. Abelian subgroups of general finite groups were studied by 
Erd\H{o}s and Straus \cite{ES76} and Pyber \cite{Pyb97}. 
%We shall return to this 
%research line in Section~\ref{subsec:questions}.

The techniques of \cite{Ols78,BGH87} are worth noting. The upper bound for 
$\alpha_\F(n, m, 2)$ is through a probabilistic argument in the linear algebra or 
geometry settings. As Pyber indicated \cite{Pyb97}, this is one of the first 
applications of the random method in group theory. The lower bound for 
$\alpha_\F(n, m, 2)$ in \cite{BGH87} relies on methods from the intersection 
theory in algebraic geometry \cite{HT84}.

% already mentioned in 
%Section~\ref{subsec:turan_bgh}, a lower bound for $\CRN_\F(s, t, 
%\ell)$ can be obtained. The proof of the following proposition is in 
%Section~\ref{subsec:proof_bgh}.

\subsection{Tur\'an meets Feldman and Propp} \label{subsec:turan_fp}
After proving his theorem in 
\cite{Tur41}, Tur\'an proposed the corresponding problem for hypergraphs 
\cite{Tur61}, which asks to determine $\TN(n, a, \ell)$ for any $\ell$. 
%Using the terminologies here, Tur\'an's problem asks to determine $\alpha(n, m, 
%\ell)$ for any $\ell$. 
This problem greatly 
stimulates the development of
extremal combinatorics; one prominent example is Razborov's invention of 
flag algebras \cite{Raz07}. More results and developments can be found in 
surveys by Sidorenko \cite{Sid95} and Keevash \cite{Kee11}.

The corresponding problem for alternating multilinear maps was studied by 
Feldman and Propp \cite{FP92}, with applications to geometry and quantum 
mechanics. 
Their main result is a lower bound of $\alpha_\F(n, m, \ell)$. This lower bound is 
more easily 
described in the form of an upper bound of the Feldman-Propp number 
$\FP_\F(a, m, \ell)$, which is a recursive function and can grow as fast as 
Ackermann's function.

Interestingly, 
Feldman and Propp termed their result as a ``linear Ramsey theorem,'' and compared 
it with the classical 
Ramsey theorem in \cite[Sec. 3]{FP92}. By the 
relations among the three numbers explained in Section~\ref{subsec:three}, this is 
a misnomer, and it is really a linear Tur\'an theorem. In particular, Feldman and 
Propp's 
argument to prove the upper bound for $\FP_\F(a, m, \ell)$ carries over to $\FP(a, 
m, 
\ell)$ in a straightforward way. 

\subsection{Tur\'an's problem for alternating multilinear maps} 
\label{subsec:turan_problem}
Let $\F$ be an 
algebraically closed field. 
%Of course, more time is 
%needed to see whether it could be as influential as the original Tur\'an's 
%problem. 
From the experience in settling $\alpha_\F(n, m, 2)$ in \cite{BGH87}, the lower 
bound derived from intersection theory matches the upper bound derived from a 
probabilistic argument. For $\alpha_\F(n, m, \ell)$, the same probabilistic 
argument, observed already in \cite{FP92}, gives an upper bound which is the 
smallest integer $a$ satisfying $n<\frac{m}{a}\cdot\binom{a}{\ell}+a$. The lower 
bound obtained from \cite{FP92} is far from this upper 
bound. It seems to us that there are certain substantial 
difficulties to generalise the intersection-theoretic calculations in 
\cite{HT84} which support the lower bound for $\alpha(n, m, 2)$ in \cite{BGH87} to 
provide better lower 
bound for $\alpha(n, m, \ell)$. Therefore, we believe that improving the lower 
bound of 
$\alpha_\F(n, m, \ell)$ for $\ell>2$ is a fascinating open problem. 
%
%was intrigued to compare their lower bound with the classical 
%Ramsey theorem \cite[Sec. 3]{FP92}. From the above discussions and the reasons 
%listed in Appendix~\ref{app:three}, we see that 
%instead of a linear Ramsey theorem, the 
%main result of \cite{FP92} would be more appropriately described as a linear 
%Tur\'an theorem. This leads to the next phase of our story: what would be a 
%natural Ramsey problem for alternating multilinear maps?

%\section{Ramsey problems for alternating multilinear maps}\label{sec:ramsey}

\subsection{Two opposite structures for totally-isotropic spaces} 
\label{subsec:opp_struct}
Let 
$\phi:(\F^n)^{\times \ell}\to\F^m$ be an alternating $\ell$-linear form. For any 
$W\leq \F^n$, the restriction of $\phi$ to $W$ naturally gives $\phi|_W:W^{\times 
\ell}\to \F^m$. Then $W$ is a totally-isotropic space if and only if $\phi|_W$ is 
the zero map, i.e. $\dim(\span(\phi(W, \dots, W)))=0$. The opposite situation then 
is 
naturally 
when $\phi|_W$ is the full map, i.e. $\dim(\span(\phi(W, \dots, 
W)))=\binom{\dim(W)}{\ell}$. 
So 
we define $W\leq \F^n$ to be a \emph{complete space} of $\phi$, if $\dim(W)\geq 
\ell$ and 
$\dim(\span(\phi(W, \dots, 
W)))=\binom{\dim(W)}{\ell}$. 

That $W$ being a totally-isotropic space for $\phi$ can also be formulated as: for 
any $w_1, \dots, w_\ell\in W$, $\phi(w_1, \dots, w_\ell)=0$. From this viewpoint, 
the opposite structure of totally-isotropic spaces would be anisotropic spaces 
defined in Section~\ref{subsec:grass}.
%that for any linearly independent $w_1, \dots, 
%w_\ell\in W$, $\phi(w_1, \dots, w_\ell)\neq 0$. If $W\leq \F^n$ of dimension 
%$\geq 
%\ell$ satisfies this 
%condition, we say that $W$ is an \emph{anisotropic space} of $\phi$. 

While it is clear that a complete space is anisotropic, the converse is not 
necessarily true. This is because there exists a non-full alternating 
$\ell$-linear map $\phi$ 
such that for any linearly independent $(v_1, \dots, v_\ell)$, $\phi(v_1, \dots, 
v_\ell)\neq 0$. When $\ell=2$ this is already seen in 
Remark~\ref{rem:compare}, and for general $\ell$ this follows easily from the 
upper bound of $\alpha_\F(n, m, \ell)$ indicated in 
Section~\ref{subsec:turan_problem}. This distinction between complete spaces and 
anisotropic 
spaces does not arise for those alternating multilinear maps constructed from 
hypergraphs, as a non-complete hypergraph certainly misses a hyperedge.
%
%\subsection{Geometric and group-theoretic interpretations.} 
%\label{subsec:interpretations}

Besides being 
natural structures opposite to totally-isotropic spaces, complete spaces and 
anisotropic spaces have known connections to group theory and geometry, 
respectively. 
%Briefly speaking, for alternating bilinear maps over $\F_p$, $p$ a 
%prime, 
%complete spaces and totally-isotropic spaces
%have natural group-theoretic interpretations. 
These connections have been explained in 
Section~\ref{sec:app} to deduce
Corollaries~\ref{cor:groups} and~\ref{cor:geometry} from
Theorem~\ref{thm:main}.

%Let us also remark that a structure closely related to complete spaces, called 
%quantum cliques, was recently studied by 
%Weaver in the context of operator systems from quantum information and operator 
%algebras \cite{Wea17,Wea19}. 

We note that a result relating complete spaces (or 
anisotropic spaces) and cliques, in the spirit of Proposition~\ref{prop:alpha} 
relating 
totally-isotropic spaces and independent sets, is not possible in general.
%Here, let us just make the remark that, while totally-isotropic spaces correspond 
%to 
%independent sets closely by 
%Proposition~\ref{prop:alpha}, complete spaces and anisotropic spaces 
%only correspond to cliques 
%superficially. 
This can be seen from the algorithmic viewpoint, say 
when $\ell=2$. 
It is well-known 
computing the maximum clique size is NP-hard on 
graphs. Computing the maximum complete space dimension can be achieved in 
randomised polynomial time, when the field size is large enough, by the 
Schwartz-Zippel lemma \cite{Sch80,Zip79}. In \cite{BCG+19}, it is shown that the 
problem of deciding whether an alternating bilinear map is anisotropic subsumes 
the problem of deciding quadratic residuosity modulo squarefree composite numbers, 
a difficult number-theoretic problem. 
%From the algorithmic perspective, to 
%derandomise this 
%algorithm based on the Schwartz-Zippel lemma is also an intriguing question. 
%This 
%distinction will be seen in 
%Theorem~\ref{thm:main} as well. 

\subsection{Ramsey numbers for alternating multilinear maps} Based on which of 
complete spaces and anisotropic spaces play the role of cliques, 
two Ramsey numbers can be defined for alternating multilinear maps.

%, in which totally-isotropic 
%spaces and complete spaces play the roles of independent sets and cliques, 
%respectively. 

\begin{definition}
For a field $\F$ and $s, t, \ell\in \N$, $s, t\geq \ell\geq 2$, 
the \emph{Ramsey number for complete spaces}, 
$\CRN_\F(s, t, \ell)$,  is the minimum number $n\in \N$ such that any alternating 
$\ell$-linear map $\phi:(\F^n)^{\times \ell}\to \F^m$ has either a 
totally-isotropic space of dimension $s$, or a complete space of dimension $t$. 

The \emph{Ramsey number for anisotropic spaces}, $\ARN_\F(s, t, \ell)$, is defined 
in the 
same way, except that 
we replace complete spaces with anisotropic spaces in the above. 
\end{definition}
It is clear that $\ARN_\F(s, t, \ell)\leq \CRN_\F(s, t, \ell)$. 

As with hypergraph Ramsey numbers, the first question is whether $\CRN_\F(s, t, 
\ell)$ and $\ARN_\F(s,t,\ell)$ are finitely upper 
bounded, and 
if so, provide an explicit bound as tight as possible. Improving bounds for 
hypergraph Ramsey numbers is a major research topic in Ramsey theory, with 
classical works by Ramsey \cite{Ram30} and Erd\H{o}s and Rado \cite{ER52}, and the 
recent breakthrough by Conlon, Fox, and Sudakov \cite{CFS10}.
%At present, we can 
%prove the following special case.
%\begin{proposition}\label{prop:ramsey}
%For any $\F$ and any $t\geq \ell\geq 1$, we have $\ARN_\F(\ell, t, \ell)$ is 
%finite, 
%\begin{equation}\label{eq:rec1}
%\ARN_\F(\ell, t, \ell)\leq \ARN_\F(\ell-1, \ARN_\F(\ell, t-1, \ell), \ell-1)+1.
%\end{equation}
%%and
%%\begin{equation}\label{eq:rec2}
%%\RN_\F(s, t, \ell)\leq 
%%\RN_\F(\RN_\F(s-1, t, \ell), \RN_\F(s, t-1, \ell), \ell-1)+1.
%%\end{equation}
%\end{proposition}
However, directly applying those methods for the hypergraph Ramsey theorem 
seems not to work. One reason is that the 
``correlation'' in the linear algebraic world prohibits the divide and conquer 
paradigm. For example, one way to prove the $\ell$-uniform hypergraph Ramsey 
theorem is to use a recursive relation for Ramsey numbers as
$$
\RN(s, t, \ell)\leq \RN(\RN(s-1, t, \ell), \RN(s, t-1, \ell), \ell-1)+1.
$$
Such a relation does not carry over to $\ARN_\F$ and $\CRN_\F$, at least not 
directly. 
The problem lies in the $\RN(s, t-1, \ell)$ and then $t-1$ branch, as we cannot 
``merge'' the two conditions into one to obtain the desired dimension-$t$ 
anisotropic or complete space. 

To the best of our knowledge, it is not even known that $\CRN$ and $\ARN$ are 
finitely bounded for general $\ell$. We therefore propose the following conjecture.
\begin{conjecture}\label{conj:multilinear}
The Ramsey numbers for complete spaces and anisotropic spaces,
$\CRN_\F(s, t, \ell)$ and $\CRN_\F(s, t, \ell)$, are upper bounded by explicit 
functions in $s$, $t$, and $\ell$.
\end{conjecture}

Our main result just proves Conjecture~\ref{conj:multilinear} in the case of 
$\ell=2$. It is also somewhat 
surprising, as it gives a \emph{polynomial upper bound} for $\CRN$, 
and therefore 
$\ARN$, in the 
case of $\ell=2$. As the reader already saw in Section~\ref{sec:thm_proof}, 
its proof strategy is very different 
from those for proving the graph Ramsey theorem.

%%% AUTHOR: optional appendix here
%\appendix %% you may comment this out if no Appendix
%\section*{Appendix}
%\section{Improving the constants}
%Material is placed here as needed.

%%% AUTHOR: optional acknowledgments here
\section*{Acknowledgments} %%  you may comment this out if no Ackno
The author would like to thank the anonymous 
reviewers 
for careful reading and thoughtful suggestions which helped to improve the 
presentation of this paper significantly. The author also would like to thank 
Yinan Li and L\'aszl\'o Pyber for helpful 
feedback, and George Glauberman for sharing his insights into \cite{BGH87}.

%%% AUTHOR:
%%% Bibliography goes here. Note that the arXiv cannot process bibtex
%%% or biber bibliographies.  Example of acceptable bibliograpy format:
\newcommand{\etalchar}[1]{$^{#1}$}
\bibliographystyle{amsplain}

%% AUTHOR: You can generate such a bibliography from a .bib file by 
%% running pdflatex/bibtex/pdflatex/pdflatex and then pasting the .bbl file
%% between \begin{thebibliography} and \end{bibliography}

%%% AUTHOR: Include a short description of each author following the
%%% structure below. Use the same short tags used previously.  
%%% Use \imageat{} and \imagedot{} instead of "@" and "." in
%%% email addresses-this replaces the symbols with graphics to avoid 
%%% e-mail address harvesting from the .pdf file
\begin{dajauthors}
\begin{authorinfo}[yq]
  Youming Qiao\\
  Centre for Quantum Software and Information \\
  University of Technology Sydney \\
  15 Broadway, Ultimo NSW 2007, Australia\\
  jimmyqiao86\imageat{}gmail\imagedot{}com \\
  \url{https://sites.google.com/site/jimmyqiao86/}
\end{authorinfo}
\end{dajauthors}

\end{document}